\documentclass[reqno,twoside,11pt]{article}
\usepackage{amssymb,amsfonts,amsthm,amsmath}
\usepackage[pagebackref=false]{hyperref}

\usepackage[hmargin=1in,vmargin=1in]{geometry}
\usepackage{cite}

\usepackage{color}

\def\eqdef{\stackrel{\rm def}{=}}

\def\beq{\begin{equation}}
\def\eeq{\end{equation}}
\def\beqs{\begin{equation*}}
\def\eeqs{\end{equation*}}

\newtheorem{theorem}{Theorem}[section]
\newtheorem{lemma}[theorem]{Lemma}

\newtheorem{corollary}[theorem]{Corollary}

\theoremstyle{definition}
\newtheorem{remark}[theorem]{Remark}

\def\myclearpage{}

\definecolor{darkred}{rgb}{.70,.12,.20}

\definecolor{darkgreen}{rgb}{.20,.52,.14}

\renewcommand{\leq}{\le}

\numberwithin{equation}{section}

\title{Generalized Forchheimer flows in heterogeneous porous media}
\author{Emine Celik and Luan Hoang}
\date{\today}

\begin{document}
\maketitle

\begin{center}
Department of Mathematics and Statistics, Texas Tech University, Box 41042, Lubbock, TX 79409--1042, U.~S.~A.\\
Email addresses:  \texttt{emine.celik@ttu.edu}, \texttt{luan.hoang@ttu.edu}\\
\end{center}

\begin{abstract}
We study the generalized Forchheimer flows of slightly compressible fluids in heterogeneous porous media.
The media's porosity and coefficients of the Forchheimer equation  are functions of the spatial variables. 
The partial differential equation for the pressure is degenerate in its gradient and can be both singular and degenerate in the spatial variables. 
Suitable weighted Lebesgue norms for the pressure, its gradient and time derivative are estimated.
The continuous dependence on the initial and boundary data is established for the pressure and its gradient with respect to those  corresponding  norms.
Asymptotic estimates  are derived even for unbounded boundary data as time tends to infinity.
\end{abstract}


\pagestyle{myheadings}\markboth{E. Celik and L. Hoang}
{Generalized Forchheimer Flows in Heterogeneous Porous Media}

\section{Introduction and preliminaries}\label{intro}

Forchheimer equations \cite{Forchh1901,ForchheimerBook} are commonly used in place of Darcy's Law to describe the fluid dynamics in porous media when the Reynolds number is large. Their nonlinear structure as opposed to the linear Darcy's equation requires new mathematical investigations.
For more thorough introduction to Forchheimer flows and their generalizations, the reader is referred to \cite{ABHI1,HI2}, see also \cite{MuskatBook,BearBook,StraughanBook,NieldBook}.

In previous articles either for incompressible fluids, e.g. \cite{Payne1999a,Straughan2010}, or compressible ones, e.g. \cite{ABHI1,HI1,HI2,HIKS1,HKP1}, the porous media considered are  always homogeneous.  In reality, however, the porous media such as soil, geological media or multi-layer media are not homogeneous. The current paper is to start our investigation of Forchheimer fluid flows in heterogeneous porous media. 
We will develop the model and analyze it mathematically. This lays the foundation for our subsequent study including maximum estimates, higher integrability of the gradient, as well as the structural stability. Such analysis will be needed in mathematical theory of homogenization and upscale computation for non-Darcy fluid flows in heterogeneous porous media.

Let a porous medium be modeled as a  bounded domain $U$ in space $\mathbb{R}^n$ with $C^1$-boundary $\Gamma=\partial U$. Throughout this paper, $n\ge 2$ even though for physics problems $n=2$ or $3$.
Let $x\in \mathbb{R}^n$ and $t\in \mathbb{R}$ be the spatial and time variables.
The porosity of this heterogeneous medium is denoted by $\phi=\phi(x)$ which  depends on the location $x$ and has values in $(0,1]$.

For a fluid flow in the media, we denote the velocity by $v(x,t)\in \mathbb{R}^n,$ pressure by $p(x,t)\in \mathbb{R}$ and density by $\rho(x,t)\in \mathbb{R}^+=[0,\infty)$.

A generalized Forchheimer equation for heterogeneous porous media is
 \beq\label{eq1}
g(x,|v|)v=-\nabla p,
\eeq
where $g(x,s)\ge 0$ is a function defined on $\bar U\times  \mathbb{R}^+$.
It  is a generalization of Darcy and Forchheimer equations \cite{ABHI1,HI1,HI2}.

The dependence  of function $g$ in \eqref{eq1} on the spatial variable $x$  is used to model the heterogeneous media. For homogeneous media, $g$ is independent of $x$.
For instance, when
\beq\label{simg} 
g(x,s)=\alpha,\
 \alpha+\beta s,\
  \alpha+\beta s+\gamma s^2,\
   \alpha +\gamma_m s^{m-1}, 
\eeq 
where $\alpha$, $\beta$, $\gamma$, $m\in(1,2]$, $\gamma_m$ are empirical constants, we have Darcy's law, Forchheimer's two term, three term and power laws, respectively, for homogeneous media, see e.g.~\cite{MuskatBook,BearBook}. 
Moreover, many models of two-term Forchheimer law obtained from experiments, see  \cite{BearBook}, have $\alpha$ and $\beta$ in \eqref{simg} depending on the porosity $\phi$.
For  heterogeneous porous media, $\phi=\phi(x)$, thus, these coefficients become functions of $x$.
This motivates the $x$-dependent model \eqref{eq1}.

In this paper, we study the model when the function $g$ in \eqref{eq1} is a generalized polynomial with non-negative coefficients. More precisely,  the function $g$ is of the form
\beq\label{eq2}
g(x,s)=a_0(x)s^{\alpha_0} + a_1(x)s^{\alpha_1}+\cdots +a_N(x)s^{\alpha_N},\quad s\ge 0,
\eeq
where $N\ge 1,\alpha_0=0<\alpha_1<\cdots<\alpha_N$ are fixed real numbers, the coefficient functions  $a_1(x)$, $a_2(x)$, \ldots, $a_{N-1}(x)$ are non-negative,  and $a_0(x),a_N(x)>0$. 
The number $\alpha_N$ is the degree of $g$ and is denoted by $\deg(g)$. Such a model \eqref{eq2} is sufficiently general to cover most examples in \cite{BearBook}.

From \eqref{eq1} one can solve for $v$ in terms of $\nabla p$ and obtain the equation
\beq\label{eq3}
v=-K(x,|\nabla p|)\nabla p,
\eeq
where the function $K:\bar U\times \mathbb{R}^+\to\mathbb{R}^+$ is defined by
\beq\label{eq4}
K(x,\xi)=\frac{1}{g(x,s(x,\xi))} \quad\text{for } x\in \bar U,\ \xi\ge 0,
\eeq
with $s=s(x,\xi)$ being the unique non-negative solution of  $sg(x,s)=\xi$.

Equation \eqref{eq3} can be seen as a nonlinear generalization of Darcy's equation.
The case when $a_i(x)$'s  are independent of $x$ was studied in depth in \cite{ABHI1,HI1,HI2,HKP1,HK1,HK2}.


In addition to \eqref{eq1} we have the equation of continuity
\beq\label{eq5}
\phi\frac{\partial \rho}{\partial t} +\nabla\cdot(\rho v)=0,
\eeq
   and the equation of state which, for (isothermal)  slightly compressible fluids, is
   \beq\label{eq6}
\frac{1}{\rho}   \frac{d\rho}{dp}=\varpi, 
\quad \text{where the compressibility } \varpi=const.>0.
   \eeq

Substituting  \eqref{eq3} and \eqref{eq5} into \eqref{eq6} we obtain a scalar partial differential equation (PDE) for the pressure:
   \beq\label{eq7}
\phi(x)   \frac{\partial p}{\partial t}=\frac1{\varpi} \nabla\cdot(K(x,|\nabla p|)\nabla p)+K(x,|\nabla p|)|\nabla p|^2.
   \eeq
   
   On the right hand side of \eqref{eq7} the constant $\varpi$ is very small for most slightly compressible fluids in porous media, hence we neglect its second term and study the following reduced equation
   \beq\label{eq8}
\phi(x)   \frac{\partial p}{\partial t}=\frac1{\varpi} \nabla\cdot(K(x,|\nabla p|)\nabla p).
   \eeq
(This simplification is  used  commonly in petroleum engineering. For a full treatment without such a simplification, see \cite{CHK1}.) 

For our mathematical study of \eqref{eq8} below, by scaling the time variable $t\to\varpi^{-1} t$, we can assume, without loss of generality,  that $\varpi=1$.


Throughout the paper, function $g(x,s)$ in \eqref{eq2} is fixed, hence so is $K(x,\xi)$.
The initial boundary value problem (IBVP) of our interest is
\begin{equation}\label{ppb}
\begin{cases}
\phi(x) \begin{displaystyle}\frac{\partial p}{\partial t}\end{displaystyle}
= \nabla \cdot (K(x,|\nabla p|)\nabla p) \quad \text{on } U\times(0, \infty),\\
p=\psi \quad \text{on} \quad \Gamma \times (0, \infty),\\
p(x,0)=p_0(x) \quad \text{on}\quad   U,
\end{cases}
\end{equation} 
where  $p_0(x)$ are $\psi(x,t)$ are given initial and boundary data.

The main goals of this paper are to estimate the solution of \eqref{ppb} in different norms and to establish its continuous dependence on the initial and boundary data. 
Regarding the PDE of \eqref{ppb}, the fact that $\phi(x)$ can be close to zero, alone, makes its left-hand side degenerate. 
In addition, as we will see in Lemma \ref{LemKx} below and the discussion right after that, the $K(x,|\nabla p|)$ is degenerate when $|\nabla p|$ is large, and can  be either very small or very large at different $x$. Therefore, we have to deal with a parabolic equation with different types of degeneracy and singularity. For the degeneracy/singularity in $x$, we use   appropriate weighted Lebesgue and Sobolev norms. 
To identify the weight functions for these norms, we carefully examine the structure of the PDE in \eqref{ppb}, particularly, the function $K(x,\xi)$. It turns out that the porosity function $\phi(x)$ and the function $W_1(x)$, which will be computed explicitly in \eqref{W12}, are the essential weights.
In order to derive differential inequalities for the weighted norms, we make use of a suitable two-weight Poincar\'e-Sobolev inequality, see \eqref{Poincarenew}.
We then proceed and obtain the estimates for the solution in section \ref{psec}, for its gradient in section \ref{GradSec}, and for its  time derivative in section \ref{ptsec}. The continuous dependence in corresponding weighted norms for both solution and its gradient are obtained in  section \ref{dependence}. The results for large time are particularly emphasized to show the long time dynamics of the problem. Their formulations are made simpler than those in  the previous works \cite{HI2,HKP1,HK1,HK2}.


For the remainder of this section, we present main properties of $K(x,\xi)$.
First, we recall some elementary inequalities that will be needed. 
Let $x,y\ge 0$, then
\beq\label{ee3}
(x+y)^p\le 2^{p-1}(x^p+y^p)\quad  \text{for all }  p\ge 1,
\eeq
\beq\label{ee5}
x^\beta \le 1+x^\gamma \quad \text{for all } \gamma\ge \beta\ge 0.
\eeq



The following exponent will be used throughout in our calculations
\beq\label{eq9}
a=\frac{\alpha_N}{\alpha_N+1}\in (0,1).
\eeq


We have from Lemmas III.5 and III.9 in \cite{ABHI1} that
\beq\label{Kderiva}
-aK(x,\xi)\le \xi \frac{\partial K(x,\xi)}{\partial\xi} \le 0\quad \forall \xi\ge 0.
\eeq
This implies  $K(x,\xi)$ is decreasing in $\xi$, hence 
\beq\label{decK}
K(x,\xi)\le K(x,0)=\frac 1{g(x,0)}=\frac 1{a_0(x)}.
\eeq

The function $K(x,\xi)$ can also be estimated from above and below in terms of $\xi$ and coefficient functions $a_i(x)$'s as follows.
Let us define the main weight functions
\beq\label{Mm}
M(x)=\max\{ a_j(x):j=0,\ldots, N\},\quad
m(x)=\min\{a_0(x),a_N(x)\},
\eeq
\beq\label{W12}
W_1(x)=\frac{a_N(x)^a}{2 N M(x)},\quad\text{and}\quad W_2(x)= \frac{NM(x)}{m(x)a_N(x)^{1-a}}.
\eeq

\begin{lemma}\label{LemKx} For $\xi\ge 0$, one has
\beq \label{KK}
\frac{2W_1(x)}{\xi^a +a_N(x)^a}\le K(x,\xi)\le \frac{W_2(x)}{\xi^a}
\eeq
and, consequently, 
\beq \label{Kwithsquare}
W_1(x)\xi^{2-a}-\frac{a_N(x)}{2} \le K(x,\xi)\xi^2\le W_2(x)\xi^{2-a}.
\eeq
\end{lemma}
\begin{proof}
Let $s=s(x,\xi)$ be defined in \eqref{eq4}. Then
\beqs
\xi=sg(x,s)=a_0(x)s+a_1(x)s^{\alpha_1+1}+\cdots +a_N(x)s^{\alpha_N+1}\ge a_N(x) s^{\alpha_N+1},
\eeqs
hence
\beq\label{sxi2}
s\le \Big(\frac\xi{a_N(x)}\Big)^{\frac 1{\alpha_N+1}}.
\eeq

Since the exponents $\alpha_j$  are increasing in $j$, then by \eqref{ee5},  one has $s^{\alpha_j}\le 1+s^{\alpha_N}$ for $j=1,\ldots,N-1$. 
Thus, we have
\beq\label{gM}
g(x,s)\le M(x)(1+s+\ldots+s^{\alpha_N})\le M(x) N (1+s^{\alpha_N}).
\eeq
Combining \eqref{sxi2} and \eqref{gM} yields
\beqs
g(x,s)\le M(x) N\Big[ 1+\Big(\frac \xi{a_N(x)}\Big)^{\frac{\alpha_N}{\alpha_N+1}}\Big]
= N M(x) \Big[ 1+\Big(\frac \xi{a_N(x)}\Big)^{a}\Big]=\frac{ N M(x)( \xi^a+a_N(x)^a)}{a_N(x)^a}.
\eeqs
Therefore
\beqs
K(x,\xi)=\frac 1{g(x,s)}\ge\frac{a_N(x)^a}{ NM(x)\big (\xi^a+a_N(x)^a\big) }=\frac{2W_1(x)}{\xi^a+a_N(x)^a},
\eeqs
which proves the first inequality in \eqref{KK}.

Now, using  \eqref{gM}
\beq\label{sxi1}
\xi=sg(x,s)  \le s\cdot M(x)N (1+s^{\alpha_N}).
\eeq

Note that $g(x,s)\ge m(x)(1+s^{\alpha_N})$, then by  \eqref{sxi1} and \eqref{sxi2}, we have 
\beqs
g(x,s) \ge \frac{m(x)\xi}{NM(x)}\cdot \frac{1}{s}
\ge \frac{m(x) \xi}{NM(x)} \cdot \frac{a_N(x)^{\frac 1{\alpha_N+1}}}{\xi^{\frac 1{\alpha_N+1}}}= \frac{m(x) a_N(x)^{1-a}\xi^a}{NM(x)}.
\eeqs
Therefore,
\beqs
K(x,\xi)=\frac 1{g(x,s)} \le \frac{NM(x)}{m(x) a_N(x)^{1-a}\xi^a} =\frac{W_2(x)}{\xi^a},
\eeqs
hence we obtain the second inequality of \eqref{KK}.

Next, multiplying \eqref{KK} by $\xi^2$, we have 
 \beq \label{K1}
\frac{2W_1(x)\xi^2}{\xi^a+a_N(x)^a} \le K(x,\xi)\xi^2\le W_2(x)  \xi^{2-a}.
\eeq
The second inequality of \eqref{K1} is exactly that of \eqref{Kwithsquare}.
For the first inequality of \eqref{Kwithsquare}, if  $\xi\ge a_N(x)$ then \eqref{K1} gives
\beq\label{K2}
K(x,\xi)\xi^2  \ge \frac{2W_1(x)\xi^2}{2\xi^a} =W_1(x)\xi^{2-a}.
\eeq
Thus, for all $\xi\ge 0$
\beq\label{K3}
K(x,\xi)\xi^2\ge W_1(x)(\xi^{2-a}-a_N(x)^{2-a})= W_1(x)\xi^{2-a}-W_1(x)a_N(x)^{2-a}.\eeq
Note that
\beq\label{W1a}
W_1(x)a_N(x)^{2-a}=\frac{a_N(x)^2}{2NM(x)}\le \frac{a_N(x)}{2N}\le \frac{a_N(x)}2.
\eeq
Hence  \eqref{K1},  \eqref{K3} and \eqref{W1a} yield the first inequality of \eqref{Kwithsquare}.
The proof is complete.
\end{proof}

We now discuss some characters of the PDE in \eqref{ppb}. 
On the left-hand side, the porosity $\phi(x)$ can be close to zero, hence giving the degeneracy in variable  $x$. 
On the right-hand side, the dependence of $K(x,\xi)$ on $\xi$ as seen in \eqref{KK} shows that the PDE is degenerate in $|\nabla p|$ as  $|\nabla p|\to\infty$.
Moreover, since the weights $W_1(x)$ and $W_2(x)$ can tend to either zero or infinity at different location $x$, then, thanks to \eqref{KK} again, so can $K(x,\xi)$.
Therefore the PDE can become singular and/or degenerate in $x$. 
The above fact about the weights $W_1(x)$ and $W_2(x)$ is supported by practical models in \cite{BearBook}.
For example, the two-term Forchheimer law (i.e. $N=1$) has  the coefficients $a_0$ and $a_1$ going to zero as $\phi\to 1$, and to infinity as $\phi \to 0$.
In this case, $\phi$ is required to be in $(0,1)$.
For heterogeneous media, constant $\phi$ becomes function $\phi=\phi(x)\in (0,1)$ and it may be close $1$ or $0$ at different values of $x$.
Therefore, thanks to the mentioned behavior of $a_0(x)$ and $a_1(x)$, the weights $W_1(x)$ and $W_2(x)$ possess the stated property.

Lastly for this section, we recall an  important monotonicity property for the PDE in \eqref{ppb}.

\begin{lemma}[c.f.~\cite{ABHI1}, Proposition III.6 and Lemma III.9]\label{mono2}
For any $y,y' \in \mathbb{R}^n$, one has 
\beqs
(K(x,|y|)y-K(x,|y'|)y')\cdot (y'-y) \ge(1-a)K(x,\max \{ |y|,|y'| \})|y-y'|^2.
\eeqs
\end{lemma}

In order to estimate the pressure gradient, similar to \cite{ABHI1,HI1,HI2}, we will make use of  the function
\beq\label{H11}
H(x,\xi)=\int_0^{\xi^2}K(x,\sqrt s )ds\quad \text{for } x\in U,\ \xi\ge 0.
\eeq
Same as (96) of \cite{ABHI1}, we have the comparison 
\beq\label{H33}
K(x,\xi)\xi^2 \le H(x,\xi)\le 2 K(x,\xi)\xi^2.
\eeq
Combining \eqref{H33} with \eqref{Kwithsquare} gives 
\beq\label{H22}
W_1(x)\xi^{2-a}-\frac{a_N(x)}2 \le H(x,\xi) \le 2W_2(x)\xi^{2-a}.
\eeq

\section{Estimates for the pressure }\label{psec}

We start analyzing the  IBVP \eqref{ppb}. To deal with the non-homogeneous boundary condition, let $\Psi(x,t)$ be an extension (in $x$) of $\psi(x,t)$ from   boundary $\Gamma$  to $U$. Our results are stated in terms of $\Psi$, but can be easily converted to $\psi$, see e.g. \cite{HI1}.

Let $\bar{p}=p-\Psi$, then we have 
\begin{align}\label{pbareqn}
\phi(x) 
\begin{displaystyle}
 \frac{\partial\bar p}{\partial t}\end{displaystyle}
&= \nabla \cdot (K(x,|\nabla p|)\nabla p)-\phi(x)\Psi_t \quad \text{on }  U\times(0, \infty),\\
\bar{p}&=0\quad\text{on }\Gamma \times (0, \infty).\nonumber
\end{align}

The analysis will make use of the following two-weight Poincar\'{e}-Sobolev inequality
\beq\label{Poincarenew}
\left( \int_U|u|^2 \phi(x) dx \right)^\frac12 \le c_P\left( \int_U W_1(x)|\nabla u|^{2-a}dx \right)^\frac{1}{2-a}
\eeq
for functions $u$ in certain classes that satisfy $u=0$ on $\Gamma$.


For some classes of functions $\phi$, $W_1$, $u$ such that the inequality \eqref{Poincarenew} is valid, see e.g. \cite{SW1992, DTP}.
Here we give a simple example that \eqref{Poincarenew} holds under the  so-called Strict Degree Condition
\beqs
\deg(g) < \frac{4}{n-2} .\tag{SDC}
\eeqs

Note that in the three dimensional cases (n=3), (SDC) reads deg$(g)<3$, hence it holds for the commonly used two-term, three-term and power Forchheimer models, see \eqref{simg}.

We recall the standard Sobolev-Poincar\'e's inequality.
Let $\mathring W^{1,q}(U)$ be the space of functions in $W^{1,q}(U)$ with vanishing trace on the boundary.
If $1\le q<n$ then 
\beq\label{PS0} \| f\|_{L^{q^*}(U)} \le c \| \nabla f\|_{L^q(U)} \text{ for all }f\in\mathring{W}^{1,q}(U),
\eeq
where the constant $c$ depends on $q$, $n$ and the domain $U$, and  
$q^*=n q/(n-q).$   

One can easily verify that  (SDC) is equivalent to $2<  (2-a)^*$.

Assume (SDC). 
Let $q<2-a$ such that $r=q^*>2$. 
Let $c$ be the positive constant in \eqref{PS0}.
Assume further that 
\beq\label{CP}
c_P\eqdef c\Big( \int_U W_1(x)^{-\frac{q}{2-a-q}} dx \Big)^\frac{2-a-q}{(2-a)q}\Big( \int_U \phi(x)^\frac{r}{r-2} dx \Big)^\frac{r-2}{2r}<\infty. 
\eeq
Let $u\in \mathring W^{1,q}(U)$. Then by H\"older's inequality  and standard Sobolev-Poincar\'e inequality \eqref{PS0}
\beqs
\Big( \int_U|u|^2 \phi(x) dx \Big)^\frac12 
\le \Big( \int_U|u|^r dx \Big)^\frac1r \Big( \int_U \phi(x)^\frac{r}{r-2} dx \Big)^\frac{r-2}{2r}  
\le c\Big( \int_U |\nabla u|^q dx \Big)^\frac1q \Big( \int_U \phi(x)^\frac{r}{r-2} dx \Big)^\frac{r-2}{2r}.
\eeqs
Since $q<2-a$, applying H\"older's inequality again to the second to last integral, we obtain \eqref{Poincarenew} with $c_P$ defined by \eqref{CP}.



\medskip

The above example shows the validity of \eqref{Poincarenew} for reasonable $\phi$, $W_1$ while $u$ belongs to a standard Sobolev space.
Nonetheless, for the purpose of this paper, it suffices to assume, without focusing on technical  weighted Sobolev spaces,  that the inequality \eqref{Poincarenew} always holds true for $u=\bar p$, as well as $u=\bar P$ in section \ref{dependence}.

\medskip
 \noindent {\bf Notation.}  
 The following notations will be used throughout the paper. 
 
$\bullet$ If $f(x)\ge 0$ is a function on $U$, then define
\beq\label{Lphi}
L^p_f(U)=\Big\{ u(x):\|u\|_{L^p_f(U)}\eqdef \Big(\int_U f(x)|u(x)|^pdx\Big)^{1/p}<\infty\Big\}.
\eeq
Notation $\|\cdot\|_{L^p_f}$ will be used as a short form of $ \|\cdot\|_{L^p_f(U)}$.

$\bullet$ We will use the symbol $C$ to denote a generic positive constant which may change its values  from place by place, may depend on the domain $U$, dimension $n$ and the Sobolev constant $c_P$ in \eqref{Poincarenew}, but not on  individual functions $a_i(x)$'s, and not on the initial data and boundary data. Constants $C_0,C_1,C_2,\ldots$ have fixed values within a proof, while  $d_1,d_2,\ldots$ are fixed positive constants throughout the paper.

$\bullet$  The notation $p_t$ stands for $\frac{\partial p}{\partial t}$. Similarly, $p_{1,t}=\frac{\partial p_1}{\partial t}$, $p_{2,t}=\frac{\partial p_2}{\partial t}$, etc.

$\bullet$ For a function $f(x,t)$, we denote $f(t)=f(\cdot,t)$.

\medskip
In this section, we  derive estimates for $\bar{p}(x,t)$ in $L^2_\phi(U)$.

\begin{lemma}
If $t>0$ then
\beq\label{dpwG0}
\frac{d}{dt} \int_U \bar{p}^2(x,t) \phi(x)  dx+\int_UK(x,|\nabla p(x,t)|)|\nabla p(x,t)|^2 dx \le C G_0(t)  ,
\eeq
where $C>0$ and 
\beq\label{G0}
 G_0(t)=B_1+\int_U a_0(x)^{-1}|\nabla \Psi(x,t)|^2 dx+ \int_U W_1(x)|\nabla \Psi(x,t)|^{2-a} dx
+\Big( \int_U |\Psi_t(x,t)|^{2} \phi(x) dx\Big)^{\frac {2-a}{2(1-a)}}
\eeq
with
\beq\label{B1}
B_1=\int_U a_N(x) dx.
\eeq
\end{lemma}
\begin{proof}
Multiplying equation \eqref{pbareqn} by $\bar{p}(x,t)$,   integrating over $U$, and using the integration by parts, we obtain 
\beqs
\frac{1}{2}\frac{d}{dt} \int_U \bar{p}^2 \phi dx=- \int_UK(x,|\nabla p|)\nabla p\cdot \nabla  \bar{p} dx - \int_U \bar{p} \Psi_t \phi dx.
\eeqs
Substituting  $\bar{p}=p-\Psi$ into the first integral on the right-hand side, and applying Cauchy-Schwarz inequality give 
\begin{multline}\label{dpsquare2}
\frac{1}{2}\frac{d}{dt} \int_U \bar{p}^2 \phi dx+\int_UK(x,|\nabla p|){|\nabla p|}^2 dx =\int_UK(x,|\nabla p|)\nabla \Psi \cdot \nabla {p} dx - \int_U \bar{p} \Psi_t \phi dx\\
\le \int_UK(x,|\nabla p|)|\nabla \Psi|  |\nabla {p}| dx + \int_U |\bar{p}| |\Psi_t| \phi dx
\eqdef I_1+I_2.
\end{multline}
%


$\bullet$ For $I_1$ in \eqref{dpsquare2}, applying Cauchy's inequality, we have 
\begin{align*}
I_1= \int_UK(x,|\nabla p|)|\nabla \Psi| \cdot |\nabla {p}| dx \le \frac 14\int_UK(x,|\nabla p|) |\nabla {p}|^2 dx
+\int_UK(x,|\nabla p|)|\nabla \Psi|^2 dx .
\end{align*}
Estimating  the last  integral by \eqref{decK}, we get 
\begin{align}\label{dpI12}
I_1\le \frac 14\int_UK(x,|\nabla p|) |\nabla {p}|^2 dx+\int_U\frac 1{a_0(x)}|\nabla \Psi|^2 dx.
\end{align}

$\bullet$ For $I_2$ in \eqref{dpsquare2}, applying H\"older's inequality and using the weighted Sobolev-Poincar\'e inequality \eqref{Poincarenew}, and applying Young's inequality with powers $2-a$ and $\frac{2-a}{1-a}$, we have 
\begin{align*}
I_2 
&\le \Big( \int_U |\bar{p}|^{2}\phi dx\Big)^{\frac 1{2}}\Big( \int_U |\Psi_t|^{2} \phi dx\Big)^{\frac 1{2}}\le c_P \Big(\int_U W_1(x)|\nabla \bar{p}|^{2-a} dx\Big)^{\frac 1{2-a}} \Big( \int_U |\Psi_t|^{2} \phi dx\Big)^{\frac 1{2}}\\
& \le 2^{a-3}  \int_U W_1(x)|\nabla \bar{p}|^{2-a} dx+C\Big( \int_U |\Psi_t|^{2} \phi dx\Big)^{\frac {2-a}{2(1-a)}}. 
\end{align*}
By triangle inequality, \eqref{ee3}, and relation \eqref{Kwithsquare} we estimate
\begin{align}
\int_U W_1(x)|\nabla \bar{p}|^{2-a}  dx
&\le 2^{1-a}\int_U (W_1(x)|\nabla p|^{2-a} +  W_1(x)|\nabla \Psi|^{2-a} )dx\nonumber \\
&\le   2^{1-a}\int_U \Big(K(|\nabla p|)|\nabla p|^{2}  +\frac{a_N(x)}2+  W_1(x)|\nabla \Psi|^{2-a} \Big) dx \nonumber \\
&\le   2^{1-a}\int_U K(|\nabla p|)|\nabla p|^{2}  dx +2^{-a} B_1+ 2^{1-a}\int_U  W_1(x)|\nabla \Psi|^{2-a} dx\label{WKbar}.
\end{align}
Therefore, 
\beq\label{dpI23}
I_2
\le \frac14  \int_UK(x,|\nabla p|)|\nabla p|^2 dx
+C B_1
+ C\int_U W_1(x)|\nabla \Psi|^{2-a} dx
+C\Big( \int_U |\Psi_t|^{2} \phi dx\Big)^{\frac {2-a}{2(1-a)}}.
\eeq
Combining \eqref{dpsquare2}, \eqref{dpI12},  and \eqref{dpI23}, we have 
\begin{multline*}
\frac{1}{2}\frac{d}{dt} \int_U\bar{p}^2 \phi  dx+\frac12\int_UK(x,|\nabla p|){|\nabla p|}^2 dx\\ 
 \le  C\int_U\frac 1{a_0(x)}|\nabla \Psi|^2 dx+CB_1
+ C\int_U W_1(x)|\nabla \Psi|^{2-a} dx
+C\Big( \int_U |\Psi_t|^{2} \phi dx\Big)^{\frac {2-a}{2(1-a)}}.
\end{multline*}
Thus we obtain \eqref{dpwG0}.
\end{proof}

For the sake of future estimates' simplicity, we  replace $B_1$ in \eqref{G0} by
\beq\label{Bstar}
B_*=\max\{B_1,1\}.
\eeq
Thus, \eqref{dpwG0} gives
\beq\label{dpwG}
\frac{d}{dt} \int_U \bar{p}^2(x,t) \phi(x)  dx+\int_UK(x,|\nabla p(x,t)|){|\nabla p(x,t)|}^2 dx \le C G(t)  ,
\eeq
where
\begin{multline}\label{G}
 G(t)=G[\Psi](t)\eqdef B_*+\int_U a_0(x)^{-1}|\nabla \Psi(x,t)|^2 dx+ \int_U W_1(x)|\nabla \Psi(x,t)|^{2-a} dx\\
+\Big( \int_U |\Psi_t(x,t)|^{2} \phi(x) dx\Big)^{\frac {2-a}{2(1-a)}}.
\end{multline}

With this change, we have 
\beq\label{Gge1}
G(t)\ge 1\quad\forall t\ge 0.
\eeq


Using \eqref{WKbar} and \eqref{dpwG}, we derive 
\begin{multline*}
\frac d{dt} \int_U \bar{p}^2 \phi dx + 2^{a-1}\int_UW_1(x)|\nabla\bar p|^{2-a}dx \\
\le \frac d{dt} \int_U \bar{p}^2 \phi dx + \int_UK(x,|\nabla p|){|\nabla p|}^2 dx+CB_*+ C\int_U W_1(x)|\nabla \Psi|^{2-a} dx\\
\le C G(t) +CB_*+ C\int_U W_1(x)|\nabla \Psi|^{2-a} dx
\le C G(t),
\end{multline*}
which, by setting $d_1=2^{a-1}$, proves 
\beq\label{withW1}
\frac d{dt} \int_U \bar{p}^2(x,t) \phi(x) dx + d_1\int_UW_1(x)|\nabla\bar p(x,t)|^{2-a}dx \le C G(t).
\eeq

Applying inequality  \eqref{Poincarenew} to $u=\bar{p}$ and utilizing it in  \eqref{withW1} give
\beq\label{difP}
\frac d{dt} \int_U \bar{p}^2(x,t)\phi(x) dx \le -d_2\Big(\int_U \bar{p}^2(x,t)\phi (x) dx \Big)^\frac{2-a}{2} + CG(t),
\eeq%
where $d_2=d_1 c_P^{a-2}$.

This nonlinear differential inequality enables us to obtain estimates for $\bar{p}$ in terms of initial and boundary data. 
They are described by the following function and numbers.

Let ${\mathcal M}(t)={\mathcal M}[\Psi](t)$ be  a continuous function on $[0,\infty)$ that satisfies
\beq\label{MG} 
{\mathcal M}(t) \text{ is increasing and } {\mathcal M}(t)\ge G(t) \ \forall t\ge 0.
\eeq

Denote
\beq\label{AB}
\mathcal A=\mathcal A[\Psi]\eqdef \limsup_{t\to\infty}G(t) \quad\text{and}\quad 
\mathcal B =\mathcal B[\Psi]\eqdef\limsup_{t\to\infty}[G'(t)]^-.
\eeq

Note from \eqref{Gge1} that 
\beq \label{MA}
{\mathcal M}(t)\ge 1\quad\forall t\ge 0, \quad\text{and }\mathcal A\ge 1.
\eeq

\medskip
\noindent \textbf{Assumptions.} Throughout the paper, we assume that each solution  $p(x,t)$ and its corresponding function $\Psi(x,t)$  have enough regularity in spatial and time variables such that all calculations are carried out legitimately.
Also, the time dependent quantities such as the above $G(t)$ and others introduced in subsequent sections are required to belong to $C([0,\infty))$, and when needed,  $C^1((0,\infty))$. The purpose of this requirement is to allow application of Gronwall's inequality and other types of estimates in Appendix \ref{append}.

\begin{theorem}\label{thm32}
\
\begin{enumerate}
\item[{\rm (i)}] If $t>0$ then 
\beq\label{EstLtwo}
\int_U 
\bar{p}^2(x,t) \phi(x)dx \le \int_U \bar{p}^2(x,0)\phi(x) dx+ C{\mathcal M}(t)^{\frac 2{2-a}}.
\eeq

\item[{\rm (ii)}] If $\mathcal A<\infty$ then 
\beq \label{noinitial}
\limsup_{t\to\infty} \int_U \bar{p}^2(x,t) \phi(x)  dx \le C \mathcal A^{\frac 2{2-a}}.
\eeq

\item[\rm{(iii)}] If $\mathcal B<\infty$ then there is $T>0$ such that for all $t>T$ 
\beq\label{EstLtwo7}
\int_U 
\bar{p}^2(x,t) \phi(x)dx \le C(\mathcal B^{\frac 1{1-a}}+ G(t)^{\frac 2{2-a}}).
\eeq
\end{enumerate}
\end{theorem}
\begin{proof}
(i) Define $y(t)=\int_U\bar{p}^2(x,t)\phi(x) dx$. We rewrite \eqref{difP} as
\beq \label{dif3}
y'(t)\le -\varphi^{-1}(y(t))+CG(t),
\eeq 
where $\varphi(z)=C_0z^\frac2{2-a}$ with $C_0=d_2^{-\frac2{2-a}}$.
Using nonlinear Gronwall's inequality in Lemma \ref{ODE3-lem}(i), we have for all $t\ge 0$ 
\beqs
y(t)\le y(0) + \varphi(C{\mathcal M}(t)),
\eeqs
hence obtaining  \eqref{EstLtwo}.

%

(ii) Applying  Lemma \ref{ODE3-lem}(ii) to the differential inequality \eqref{dif3}, we get
\beqs
\limsup_{t\to\infty} y(t)
\le C \limsup_{t\to\infty} G(t)^\frac2{2-a}
=C \mathcal A^{\frac 2{2-a}},
\eeqs
which proves  \eqref{noinitial}.

(iii) Finally,  note that $\varphi(z)\le \varphi_0(z)\eqdef C_0(z+z^\gamma)$ for  $z\ge 0$, where  $1<\gamma=\frac{2}{2-a}<2$. 

Clearly, $\varphi_0^{-1}(y)\le \varphi^{-1}(y)$, then we have  from \eqref{dif3} that
$$ y'(t)\le-\varphi_0^{-1}(y(t))+CG(t).$$
 Hence by Lemma \ref{supinf}, there is $T>0$ such that for all $t>T$
\beqs
y(t)\le C(1+\mathcal B^\frac{\gamma}{2-\gamma}+G(t)^\gamma),
\eeqs
which, together with \eqref{Gge1},  yields \eqref{EstLtwo7}.
\end{proof}

 
\section{Estimates for the pressure's gradient}\label{GradSec}
In this section we estimate the weighted $L^{2-a}_{W_1}$-norm for the gradient of $p$.
Due to the structure of equation \eqref{pbareqn}, we start with estimates for  $H(x,|\nabla p(x,t)|)$ defined by \eqref{H11}, and will use relation \eqref{H22} to derive the ones desired.
We define 
\beq\label{G8}
G_1(t)=G_1[\Psi](t)\eqdef \int_U a_0(x)^{-1} |\nabla \Psi_t(x,t)|^2 dx .
\eeq

\begin{theorem}\label{graddest1}
\
\begin{enumerate}
\item[\rm (i)] For $t>0$,
\begin{multline}\label{gradpW111}
\int_U H(x,|\nabla p(x,t)|)dx \le e^{-\frac14t} \int_U H(x,|\nabla p(x,0)|)dx\\
+ C\Big(\int_U \bar{p}^2(x,0)\phi(x)  dx+\mathcal M^\frac2{2-a}(t) +\int_0^t e^{-\frac14(t-\tau)}G_1(\tau)  d\tau\Big).
\end{multline}

\item[\rm (ii)] If $\mathcal A<\infty$ then 
\begin{equation}\label{gradpWone}
\limsup_{t\to\infty}\int_U H(x,|\nabla p(x,t)|)dx  \le C\Big(\mathcal A^\frac 2{2-a} + \limsup_{t\to\infty} G_1(t)\Big).
\end{equation}
\end{enumerate}
\end{theorem}
\begin{proof}
(i) Multiplying  equation \eqref{pbareqn} by $\bar p_t$, integrating over $U$,  and using integration by parts we have
\begin{align*}
\int_U \bar p_t^2\phi dx &=-\int_U K(x,|\nabla p|)\nabla p\cdot \nabla \bar p_t dx -\int_U  \bar p_t  \Psi_t \phi dx\\
 &=  - \int_U K(x,|\nabla p|)\nabla p\cdot \nabla  p_t dx
+ \int_U K(x,|\nabla p|)\nabla p\cdot \nabla \Psi_t dx-\int_U  \bar p_t \Psi_t \phi dx\\
&=  - \frac 12 \frac d{dt}\int_U H(x,| \nabla p|)dx
+ \int_U K(x,|\nabla p|)\nabla p\cdot \nabla \Psi_t dx-\int_U  \bar p_t \Psi_t \phi dx.
\end{align*}
Let $\varepsilon>0$. Applying Cauchy's inequality, we derive 
\begin{multline*}
\int_U \bar p_t^2 \phi dx + \frac12\frac{d}{dt} \int_U H(x,|\nabla p|) dx
\le\varepsilon \int_U K(x,|\nabla p|)|\nabla p|^2 dx+\frac 1{4\varepsilon} \int_U K(x,|\nabla p|) |\nabla \Psi_t|^2 dx\\
 + \frac 12\int_U \bar p_t^2 \phi dx +\frac12\int_U |\Psi_t|^2 \phi dx.
\end{multline*}
By using \eqref{H33} to estimate the first term on the right-hand side, and using \eqref{decK}  to estimate the second term on the right-hand side, we obtain 
\begin{multline*}
\frac 12\int_U \bar p_t^2 \phi dx + \frac12\frac{d}{dt} \int_U H(x,|\nabla p|) dx\\
\le\varepsilon \int_U H(x,|\nabla p|)dx +\frac 1{4\varepsilon} \int_U a_0(x)^{-1} |\nabla \Psi_t|^2 dx + \frac12\int_U |\Psi_t|^2 \phi  dx.
\end{multline*}
Thus,
\beq\label{H1}
\int_U \bar p_t^2\phi dx + \frac{d}{dt} \int_U H(x,|\nabla p|) dx
\le  2\varepsilon \int_U H(x,|\nabla p|)dx + \frac1{2\varepsilon} G_1(t)+G(t).
\eeq
From  \eqref{dpwG} and \eqref{H33}, we have
\begin{equation}\label{H2}
\frac d{dt} \int_U \bar{p}^2\phi dx + \frac{1}2\int_U H(x,|\nabla p|) dx \le C G(t).
\end{equation}
Combining \eqref{H1} and \eqref{H2} with  $\varepsilon=1/8$, we obtain
\begin{multline}\label{ptrowest}
\frac d{dt}\int_U \bar{p}^2\phi dx +\frac d{dt} \int_U H(x,|\nabla p(x,t)|)dx 
+ \int_U \bar p_t^2 \phi dx  + \frac14\int_U H(x,|\nabla p(x,t)|)dx\\
\le   C (G(t)+G_1(t)).
\end{multline}
We rewrite the first term on the left-hand side and apply Cauchy's inequality as follows 
\begin{equation*}
\frac d{dt} \int_U \bar{p}^2\phi dx=2\int_U \bar{p}\bar{p}_t \phi dx
\ge - \frac12 \int_U \bar p_t^2 \phi dx-2\int_U \bar{p}^2\phi dx,
\end{equation*}
hence
\begin{equation}\label{dH0}
\frac d{dt} \int_U H(x,|\nabla p(x,t)|)dx+   \frac12\int_U\bar p_t^2\phi dx+ \frac14\int_U H(x,|\nabla p(x,t)|)dx
\le 2 \int_U \bar{p}^2\phi dx+  C (G(t)+ G_1(t)).
\end{equation}
Particularly, neglecting the second integral of the left-hand side reduces \eqref{dH0} to
\begin{equation}\label{dH}
\frac d{dt} \int_U H(x,|\nabla p(x,t)|)dx+ \frac14  \int_U H(x,|\nabla p(x,t)|)dx
\le 2 \int_U \bar{p}^2\phi dx+  C (G(t)+ G_1(t)).
\end{equation}

Using \eqref{EstLtwo} to estimate the integral term on the right-hand side, and then properties \eqref{MG}, \eqref{MA}, we have 
\begin{align*}
& \frac d{dt} \int_U H(x,|\nabla p(x,t)|)dx  +\frac14 \int_U H(x,|\nabla p(x,t)|)dx\\
& \le  2 \int_U \bar{p}^2(x,0) \phi dx+ C{\mathcal M}(t)^{\frac 2{2-a}}+C (G(t)+ G_1(t))
 \le  C \int_U \bar{p}^2(x,0) \phi dx+ C{\mathcal M}(t)^{\frac 2{2-a}}+C G_1(t).
\end{align*}

Consequently, by Gronwall's inequality,
\begin{multline}\label{nop0}
 \int_U H(x,|\nabla p(x,t)|)dx \le e^{-\frac14 t} \int_U H(x,|\nabla p(x,0)|)dx\\
+ C\int_0^t e^{-\frac14(t-\tau)}\Big(\int_U \bar{p}^2(x,0) \phi dx+ {\mathcal M}(\tau)^{\frac 2{2-a}} +   G_1(\tau)\Big) d\tau.
\end{multline}
Since ${\mathcal M}(\tau)\le {\mathcal M}(t)$ for all $\tau\in[0,t]$, estimate \eqref{gradpW111}  follows \eqref{nop0}.

(ii) Applying Lemma \ref{ODE3-lem}(ii) to differential inequality \eqref{dH}, and using limit estimate \eqref{noinitial}, we have
\begin{align*}
\limsup_{t\to\infty} \int_U H(x,|\nabla p(x,t)|)dx  
&\le  C \limsup_{t\to\infty}  \int_U \bar{p}^2\phi dx+  C \limsup_{t\to\infty}  (G(t)+G_1(t))\\
&\le  C(\mathcal A^\frac2{2-a}  + \mathcal A + \limsup_{t\to\infty}  G_1(t)).
\end{align*}
Since $\mathcal A\ge 1$, by \eqref{MA}, we obtain \eqref{gradpWone}.
\end{proof}

For large time, we improve the  estimates in Theorem \ref{graddest1} by establishing inequalities of uniform Gronwall-type  \cite{TemamDynBook,SellYouBook}.

\begin{lemma}\label{lemHuni}
For $t\ge 1$,
\begin{multline}\label{Gtwo}
\int_U H(x,|\nabla p(x,t)|)dx+\frac{1}{2}\int_{t-\frac 12}^t\int_U \bar{p}_t^2(x,\tau)\phi(x) dxd\tau\\
\le C\Big(\int_U \bar{p}^2(x,t-1)\phi(x) dx+\int_{t-1}^t (G(\tau)+G_1(\tau))d\tau\Big).
\end{multline}
\end{lemma}
\begin{proof}
The proof follows \cite{HI2} by using basic differential inequalities \eqref{H1} and \eqref{H2}.
 
Integrating \eqref{H2}  from $t-1$ to $t$ yields 
\begin{align*}
\int_U\bar{p}^2(x,t)\phi dx+ \frac12\int_{t-1}^t\int_UH(x,|\nabla p(x,\tau)|)dxd\tau\le \int_U\bar{p}^2(x,t-1)\phi dx+C\int_{t-1}^tG(\tau)d\tau.
\end{align*}
Neglecting first term on the left-hand side, we have 
\begin{equation}\label{Gone}
\int_{t-1}^t \int_U H(x,|\nabla p(x,\tau)|)dxd\tau \le 2\int_U \bar{p}^2(x,t-1)\phi dx+C\int_{t-1}^t G(\tau)d\tau.
\end{equation}
Using \eqref{H1} with $\varepsilon =1/2$ gives 
\beqs
\int_U \bar p_t^2(x,t) \phi dx+  \frac{d}{dt} \int_U H(x,|\nabla p(x,t)|) dx\le \int_U H(x,|\nabla p(x,t)|)dx + C(G(t)+G_1(t)).
\eeqs
Let $s\in[t-1, t]$. Integrating the previous inequality in time from $s$ to $t$, we obtain 
\begin{align*}
&\int_s^t\int_U \bar p_t^2(x,\tau)\phi dxd\tau +   \int_U H(x,|\nabla p(x,t)|) dx\\
&\le\int_U H(x,|\nabla p(x,s)|) dx+ \int_s^t\int_U H(x,|\nabla p(x,\tau)|)dxd\tau + C\int_s^t(G(\tau)+G_1(\tau))d\tau\\
&\le \int_U H(x,|\nabla p(x,s)|) dx+ \int_{t-1}^t\int_U H(x,|\nabla p(x,\tau)|)dxd\tau + C\int_{t-1}^t (G(\tau)+G_1(\tau))d\tau.
\end{align*}
Integrating the last  inequality in $s$ from $t-1$ to $t$ results in 
\begin{multline*}
\int_{t-1}^t\int_s^t\int_U \bar p_t^2(x,\tau)\phi dxd\tau ds + \int_U H(x,|\nabla p(x,t)|) dx\\
\le 2\int_{t-1}^t\int_U H(x,|\nabla p(x,\tau)|)dxd\tau  +C\int_{t-1}^t (G(\tau)+G_1(\tau))d\tau .
\end{multline*}
Using \eqref{Gone} to estimate the first term on the right-hand side, we have 
\begin{multline}\label{ptHtmin1}
\int_{t-1}^t\int_s^t\int_U \bar p_t^2(x,\tau)\phi dxd\tau ds + \int_U H(x,|\nabla p(x,t)|) dx\\
\le 4\int_U \bar{p}^2(x,t-1)\phi dx+C\int_{t-1}^t (G(\tau)+G_1(\tau))d\tau.
\end{multline}
For the first integral on the left-hand side, we observe that 
\beqs
\int_{t-1}^t\int_s^t\int_U \bar p_t^2(x,\tau)\phi dxd\tau ds
\ge \int_{t-1}^{t-1/2}  \int_{t-\frac 12}^t \int_U \bar p_t^2(x,\tau)\phi dx  d\tau ds
= \frac 12 \int_{t-\frac 12}^t\int_U \bar p_t^2(x,\tau)\phi dxd\tau.
\eeqs
Utilizing this estimate in \eqref{ptHtmin1}, we obtain inequality \eqref{Gtwo}.
\end{proof}

Combining Lemma \ref{lemHuni} with Theorem \ref{thm32} results in the following specific estimates.

\begin{theorem}\label{theo33}
\
\begin{enumerate}
\item[\rm(i)] If $t\ge 1$ then 
\beq\label{newgrad}
\int_U H(x,|\nabla p(x,t)|)dx\le C\Big(\int_U \bar{p}^2(x,0)\phi(x) dx+{\mathcal M}(t)^\frac{2}{2-a} + \int_{t-1}^t G_1(\tau)d\tau\Big).
\eeq

\item[\rm(ii)] If $\mathcal A<\infty$ then 
\beq\label{Gtwo5}
 \limsup_{t\to \infty} \int_U H(x,|\nabla p(x,t)|)dx\le   C\Big(\mathcal A^{\frac 2{2-a}}+  \limsup_{t\to \infty} \int_{t-1}^tG_1(\tau) d\tau\Big).
\eeq

\item[\rm(iii)] If $\mathcal B<\infty$  then there is $T>1$ such that  for all $t>T$,
\beq\label{Gtwo33}
\int_U H(x,|\nabla p(x,t)|)dx
\le  C\Big(\mathcal B^{\frac 1{1-a}}+ G(t)^{\frac 2{2-a}}+\int_{t-1}^t G_1(\tau)d\tau\Big).
\eeq
\end{enumerate}
\end{theorem}
\begin{proof}
(i) Combining \eqref{Gtwo}  with estimate \eqref{EstLtwo} and  property \eqref{MG} yields
\begin{align*}
\int_U H(x,|\nabla p(x,t)|)dx
&\le C\Big(\int_U \bar{p}^2(x,0)\phi(x) dx+ {\mathcal M}(t-1)^{\frac 2{2-a}}+\int_{t-1}^t (G(\tau)+G_1(\tau))d\tau\Big)\\
&\le C\Big(\int_U \bar{p}^2(x,0)\phi(x) dx+ {\mathcal M}(t)^{\frac 2{2-a}}+ {\mathcal M}(t)+\int_{t-1}^t G_1(\tau)d\tau\Big).
\end{align*}
Then using the fact ${\mathcal M}(t)\ge 1$ from \eqref{MA}, we obtain \eqref{newgrad}.

(ii) Taking limit superior of \eqref{Gtwo}, and using limit estimate  \eqref{noinitial},  we have
\begin{align*}
 \limsup_{t\to \infty} \int_U H(x,|\nabla p|)(x,t)dx
 &\le   C\limsup_{t\to\infty} G(t)^{\frac 2{2-a}}+ C \limsup_{t\to \infty} \int_{t-1}^t [G(\tau)+G_1(\tau)]d\tau.
\end{align*}
Note that 
\beq\label{intlim}
\limsup_{t\to \infty} \int_{t-1}^t G(\tau)d\tau \le \limsup_{t\to\infty} G(t).
\eeq
Then 
\begin{align*}
 \limsup_{t\to \infty} \int_U H(x,|\nabla p|)(x,t)dx
&\le C\Big(\mathcal A^\frac2{2-a}+\mathcal A+ \limsup_{t\to \infty} \int_{t-1}^tG_1(\tau)d\tau\Big).
\end{align*}

Estimate   \eqref{Gtwo5} then follows since $\mathcal A\ge 1$.

(iii) Using \eqref{EstLtwo7} to estimate the term $\int_U\bar{p}^2(x,t-1)\phi dx$ term in \eqref{Gtwo}, we obtain 
\beq\label{old}
\int_U H(x,|\nabla p(x,t)|)dx
\le C\Big(\mathcal B^{\frac 1{1-a}}+ G(t-1)^{\frac 2{2-a}}+\int_{t-1}^t (G(\tau)+G_1(\tau))d\tau\Big).
\eeq
Note from Lemma \ref{Btt} that
\beqs
G(\tau)\le G(t)+\mathcal B+1\quad \forall \tau\in[t-1,t].
\eeqs
Hence \eqref{old} implies
\beqs
\int_U H(x,|\nabla p(x,t)|)dx
\le C\Big(\mathcal B^{\frac 1{1-a}}+ (G(t)+\mathcal B+1)^{\frac 2{2-a}}+(G(t)+\mathcal B+1)+\int_{t-1}^t G_1(\tau)d\tau\Big).
\eeqs
Then inequality \eqref{Gtwo33} follows by using \eqref{ee3}, \eqref{ee5} and the fact $G(t)\ge 1$.
\end{proof}
\begin{remark}
(a) Compared to \eqref{gradpW111}, estimate \eqref{newgrad} does not require $\nabla p(x,0)$.
Also, \eqref{Gtwo5} improves \eqref{gradpWone} slightly, particularly when $G_1(t)$  fluctuates strongly in time. 
(b) The estimate \eqref{Gtwo33} is simpler than \eqref{old} which is the form usually presented in previous papers \cite{HI2,HKP1,HK1,HK2}.
\end{remark}

The statements in Theorems \ref{graddest1} and \ref{theo33} can be rewritten to give estimates for the integral
$\int_U W_1(x)|\nabla p(x,t)|^{2-a} dx $, that is, $\|\nabla p(t)\|_{L^{2-a}_{W_1}}^{2-a}$.

\begin{corollary}\label{cor35}
For $t>0$,
\begin{multline}\label{gradpW1}
\int_U W_1(x)|\nabla p(x,t)|^{2-a} dx \le e^{-\frac14t} \int_U H(x,|\nabla p(x,0)|)dx\\
+ C\Big(\int_U \bar{p}^2(x,0)\phi(x)  dx+\mathcal M^\frac2{2-a}(t) +\int_0^t e^{-\frac14(t-\tau)}G_1(\tau)  d\tau\Big).
\end{multline}

For $t\ge 1$,
\beq\label{C1}
 \int_U W_1(x)|\nabla p(x,t)|^{2-a} dx \le C\Big(\int_U \bar{p}^2(x,0)\phi(x) dx+{\mathcal M}(t)^\frac{2}{2-a} + \int_{t-1}^t G_1(\tau)d\tau\Big).
\eeq

If $\mathcal A<\infty$ then 
\begin{equation}\label{C3}
\begin{aligned}
 &\limsup_{t\to\infty}\int_U W_1(x)|\nabla p(x,t)|^{2-a} dx  \le C\Big(\mathcal A^\frac 2{2-a} + \limsup_{t\to\infty} G_1(t)\Big).
\end{aligned}
\end{equation}

If $\mathcal B<\infty$ then there is $T>1$ such that  for all $t>T$,
\beq\label{C4}
\int_U W_1(x)|\nabla p(x,t)|^{2-a} dx
\le  C\Big(\mathcal B^{\frac 1{1-a}}+ G(t)^{\frac 2{2-a}}+\int_{t-1}^t G_1(\tau)d\tau\Big).
\eeq
\end{corollary}
\begin{proof}
Using property \eqref{H22}, definitions \eqref{B1} and \eqref{Bstar} we have
\beqs
\int_U W_1(x)|\nabla p(x,t)|^{2-a} dx \le \int_U \Big[\frac{a_N(x)}2 + H(x,|\nabla p(x,t)|)\Big]  dx
\le \frac{B_*}2+ \int_U H(x,|\nabla p(x,t)|)  dx.
\eeqs
Also, from definition \eqref{G}, $G(t)\ge B_*$.
With these  relations, the estimates \eqref{gradpW1}, \eqref{C1}, \eqref{C3}, \eqref{C4} immediately follow \eqref{gradpW111}, \eqref{newgrad}, \eqref{Gtwo5}, \eqref{Gtwo33}, respectively.
\end{proof}


\section{Estimates for the pressure's time derivative}\label{ptsec}
In this section, we  estimate the pressure's time derivative. Let 
$$q(x,t)=p_t(x,t)\quad \text{and}\quad  \bar{q}(x,t)=\bar{p}_t(x,t)=p_t(x,t)-\Psi_t.$$ 
Then $\bar{q}$ solves 
\begin{align} \label{Qt}
\phi(x)\begin{displaystyle} \frac{\partial \bar{q}}{\partial t}\end{displaystyle}
&=\nabla \cdot(K(x,|\nabla p |)\nabla p)_t-\phi(x)\Psi_{tt} \quad \text{on } U\times(0,\infty), \\ 
\bar{q} &=0 \quad  \text{on }  \Gamma \times (0,\infty). \nonumber
\end{align}

In the following estimates, we use 
\beq\label{Gsix}
G_2(t)=G_2[\Psi](t)\eqdef\int_U|\Psi_{tt}(x,t)|^2 \phi(x) dx.
\eeq

\begin{lemma}
One has for $t>0$ and $\varepsilon>0$ that
\begin{multline}\label{qest1}
 \frac d{dt}\int_U\bar q^2(x,t)\phi(x) dx
\le-(1-a)\int_U K(x,|\nabla p(x,t) |)|\nabla q(x,t)|^2dx\\
+\varepsilon\int_U|\bar q(x,t)|^2 \phi(x) dx+CG_1(t)+C\varepsilon^{-1} G_2(t).
\end{multline}
\end{lemma}
\begin{proof}
Multiplying \eqref{Qt} by $\bar q$, integrating over $U$, and using integration by parts  we have
\begin{align*}
\int_U\frac{\partial \bar{q}}{\partial t}\bar q\phi dx
&=\int_U \nabla \cdot(K(x,|\nabla p |)\nabla p)_t \bar q dx
-\int_U \Psi_{tt}\phi)\bar q  dx\\
&=-\int_U (K(x,|\nabla p |)\nabla p)_t\cdot\nabla\bar qdx-\int_U\Psi_{tt}\bar q \phi dx.
\end{align*}
Taking the derivative in $t$ for the first integral on the right-hand side, we derive 
\beqs
\frac 12 \frac d{dt}\int_U\bar q^2\phi dx
= -\int_U \frac{\partial K(x,|\nabla p |)}{\partial \xi}\frac {(\nabla p\cdot\nabla q)}{|\nabla p|}(\nabla p\cdot\nabla\bar q)dx
-\int_U K(x,|\nabla p |)\nabla q\cdot\nabla\bar qdx-\int_U\Psi_{tt}\bar q \phi dx.
\eeqs
Using the fact that  $\bar{q}(x,t)=q(x,t)-\Psi_t$ for the first two integrals on the right-hand side, we rewrite 
\begin{align*}
\frac 12 \frac d{dt}\int_U\bar q^2\phi dx
&= -\int_U \frac{\partial K(x,|\nabla p |)}{\partial \xi} \frac{|\nabla p\cdot\nabla q|^2}{|\nabla p|}dx 
+\int_U \frac{\partial K(x,|\nabla p |)}{\partial \xi} \frac{\nabla p\cdot\nabla q}{|\nabla p|}\nabla p\cdot\nabla\Psi_tdx\\
&\quad-\int_U K(x,|\nabla p |)|\nabla q|^2dx+\int_U K(x,|\nabla p |)\nabla q\cdot\nabla \Psi_tdx
-\int_U\Psi_{tt}\bar q \phi dx.
\end{align*}
Next, by derivative property  \eqref{Kderiva} and Cauchy-Schwarz's inequality we obtain 
\begin{align*}
&\frac 12 \frac d{dt}\int_U\bar q^2\phi dx
\le a\int_U K(x,|\nabla p |)|\nabla q|^2dx+a\int_U K(x,|\nabla p |)|\nabla q||\nabla\Psi_t|dx\\
&\quad-\int_U K(x,|\nabla p |)|\nabla q|^2dx+\int_U K(x,|\nabla p |)|\nabla q||\nabla \Psi_t|dx-\int_U\Psi_{tt}\bar q\phi dx\\
&\le -(1-a)\int_U K(x,|\nabla p |)|\nabla q|^2dx+(a+1)\int_U K(x,|\nabla p |)|\nabla q||\nabla\Psi_t|dx+\int_U|\Psi_{tt}| |\bar q| \phi dx.
\end{align*}
Let $\varepsilon'>0$.
Applying Cauchy's inequality to the last two  integrals gives
\begin{align*}
\frac 12 \frac d{dt}\int_U\bar q^2\phi dx
&\le(\varepsilon'(a+1)-(1-a))\int_U K(x,|\nabla p |)|\nabla q|^2dx\\
&\quad +\frac{a+1}{4\varepsilon'}\int_U K(x,|\nabla p |)|\nabla\Psi_t|^2dx
+\varepsilon\int_U|\bar q|^2 \phi dx+\frac{1}{4\varepsilon}\int_U|\Psi_{tt}|^2 \phi dx.
\end{align*}
We estimate $K(x,|\nabla p|)$ in the second integral on the right-hand side by \eqref{decK}, then it follows
\begin{align*}
\frac 12 \frac d{dt}\int_U\bar q^2\phi dx
&\le(\varepsilon'(a+1)-(1-a))\int_U K(x,|\nabla p |)|\nabla q|^2dx\\
&\quad+\frac{a+1}{4\varepsilon'}\int_U a_0(x)^{-1}|\nabla\Psi_t|^2dx
+\varepsilon\int_U|\bar q|^2 \phi dx+\frac{G_2(t)}{4\varepsilon}.
\end{align*} 
Selecting  $\varepsilon'=(1-a)/(2(1+a))$  gives
\beqs 
\frac12 \frac d{dt}\int_U\bar q^2\phi dx\le-\frac{1-a}2\int_U K(x,|\nabla p |)|\nabla q|^2dx+CG_1(t)+\varepsilon\int_U|\bar q|^2 \phi dx+\frac {C G_2(t)}\varepsilon ,
\eeqs
which proves \eqref{qest1}.
\end{proof}

The next theorem contains different estimates of $\|\bar p_t(t)\|_{L_\phi^2}$ for both small and large time in terms of the initial and boundary data.
Note that we cannot estimate the norm at $t=0$.

\begin{theorem}
 \
\begin{enumerate}
\item[\rm(i)] For $t_0\in(0,1]$ and $t\ge t_0$,
\begin{multline}\label{greatert0}
 \int_U H(x,|\nabla p(x,t)|)dx+\int_U\bar p_t^2(x,t)\phi(x) dx
\le C\Big(t_0^{-1} \int_U \Big[H(x,|\nabla p(x,0)|)+\bar p^2(x,0)\phi(x)\Big] dx\\
 + t_0^{-1}\int_0^{t_0} G_1(\tau)d\tau + {\mathcal M}(t)^\frac{2}{2-a} + \int_{0}^t e^{-\frac14(t-\tau)} (G_1(\tau)+G_2(\tau))d\tau\Big).
\end{multline}

\item[\rm(ii)] If $t\ge 1$ then 
\beq\label{G86}
 \int_U\bar p_t^2(x,t)\phi(x) dx \le C\Big( \int_U \bar{p}^2(x,0)\phi(x) dx+ {\mathcal M}(t)^{\frac 2{2-a}}  +\int_{t-1}^t (G_1(\tau)+G_2(\tau))d\tau\Big).
\eeq

\item[\rm(iii)] If $\mathcal A<\infty$ then 
\beq\label{Gthree3}
  \limsup_{t\to \infty} \int_U\bar p_t^2(x,t)\phi(x) dx \le  C\Big(\mathcal A^{\frac 2{2-a}}
+ \limsup_{t\to \infty} \int_{t-1}^t (G_1(\tau)+G_2(\tau)) d\tau\Big).
\eeq
Consequently,
\beq\label{limsupq1}
\limsup_{t\to\infty} \int_U \bar p_t^2(x,t)\phi(x) dx
\le C\Big(\mathcal A^\frac2{2-a} +  \limsup_{t\to\infty} (G_1(t)+G_2(t))\Big).
\eeq

\item[\rm(iv)] If $\mathcal B<\infty$ then  there is $T>0$ such that for all $t>T$,
\beq\label{Gthree11}
 \int_U\bar p_t^2(x,t)\phi(x) dx \le C\Big(\mathcal B^{\frac 1{1-a}}+ G(t)^{\frac 2{2-a}}+\int_{t-1}^t (G_1(\tau)+G_2(\tau))d\tau\Big).
\eeq
\end{enumerate}
\end{theorem}
\begin{proof}
Denote $I(t)=\int_U H(x,|\nabla p(x,t)|)dx+\int_U\bar q^2(x,t)\phi(x) dx$ for $t>0$.

(i) Adding \eqref{dH0} to \eqref{qest1} with  $\varepsilon=1/4$ yields
\beq\label{dHq}
\frac d{dt} I(t) + \frac14 I(t)\le C\int_U \bar{p}^2\phi dx+CG_3(t),\quad \text{where }G_3=G+G_1+G_2.
\eeq

Integrating  \eqref{ptrowest} in time from $0$ to $t$, we have 
\beq\label{qtimeint}
 \int_0^t I(\tau)  d\tau\le  C_1 I(0) +C_1\int_0^t (G(\tau)+G_1(\tau))d\tau
\eeq
for some $C_1>0$.  Applying \eqref{qtimeint} to $t=t_0$, then there exists $t_*\in (0,t_0)$ such that
\beq\label{mvt}
 I(t_*) \le  \frac {2C_1}{t_0} I(0) +\frac {2C_1}{t_0} \int_0^{t_0} (G(\tau)+G_1(\tau))d\tau.
\eeq

For $t\ge t_0$, applying  Gronwall's inequality to \eqref{dHq} on the interval $[t_*,t]$, and then combining with \eqref{mvt}, we 
have
\begin{align*}
I(t)
&\le e^{-\frac14(t-t_*)} I(t_*) +C\int_{t_*}^t e^{-\frac14(t-\tau)}\Big [\int_U \bar{p}^2(x,\tau)\phi(x) dx+G_3(\tau)\Big]d\tau\\
& \le \frac{ C }{t_0}I(0) +\frac{ C }{t_0}\int_0^{t_0} (G(\tau)+G_1(\tau))d\tau
 +C\int_{0}^t e^{-\frac14 (t-\tau)}\Big [\int_U \bar{p}^2(x,\tau)\phi(x) dx+G_3(\tau)\Big]d\tau.
\end{align*}
Using \eqref{EstLtwo} to estimate the  integral $\int_U \bar{p}^2(x,\tau)\phi(x) dx$ yields
\begin{align*}
I(t)
&\le \frac{ C }{t_0} I(0) +\frac{ C }{t_0} \int_0^{t_0} (G(\tau)+G_1(\tau))d\tau\\
&\quad +C\int_{0}^t e^{-\frac14 (t-\tau)} \Big[\int_U \bar{p}^2(x,0)\phi(x) dx+{\mathcal M}(\tau)^\frac2{2-a}+G_3(\tau)\Big]d\tau.
\end{align*}
Since $\mathcal M(t)$ is increasing, see \eqref{MG}, it follows that 
\begin{multline}\label{It}
I(t)
\le \frac{ C }{t_0} I(0) +\frac{ C }{t_0} \int_0^{t_0} G_1(\tau)d\tau +C\mathcal M(t_0) + C \int_U \bar{p}^2(x,0)\phi(x) dx\\
 +{\mathcal M}(t)^\frac2{2-a} +C\int_{0}^t e^{-\frac14 (t-\tau)} G_3(\tau)d\tau.
\end{multline}
Also, from  \eqref{MA}, $\mathcal M(t)\ge 1$, then $\mathcal M(t_0)\le \mathcal M(t)\le \mathcal M(t)^\frac2{2-a}$.
Thus, we obtain estimate \eqref{greatert0} from \eqref{It}.

(ii) Using \eqref{qest1} with $\varepsilon=1/2$, and dropping first integral on the right-hand side, we have 
\beq\label{dq}
\frac 12 \frac d{dt}\int_U\bar q^2\phi dx\le \frac 12\int_U \bar q^2\phi  dx+C_2(G_1(t)+G_2(t))
\eeq
for some $C_2>0$.
For $s\in[t-\frac 12, t]$, integrating \eqref{dq} in time from $s$ to $t$ gives
\begin{align*}
\frac 12 \int_U\bar q^2(x,t)\phi dx
&\le \frac 12 \int_U\bar q^2(x,s)\phi dx+\frac 12\int_s^t\int_U|\bar q(x,\tau)|^2\phi  dxd\tau+ C_2\int_s^t (G_1(\tau)+G_2(\tau))d\tau\\
& \le \frac 12 \int_U\bar q^2(x,s)\phi dx+\frac 12\int_{t-\frac 12}^t\int_U|\bar q(x,\tau)|^2 \phi dxd\tau+ C_2\int_{t-1}^t  (G_1(\tau)+G_2(\tau))d\tau.
\end{align*}
Next, integrating in $s$ from $t-\frac 12$ to $t$, we have 
\begin{multline*}
\frac 12\cdot \frac 12 \int_U\bar q^2(x,t)\phi dx
 \le \frac 12  \int_{t-\frac 12}^t\int_U\bar q^2(x,s)\phi dxds
+\frac 12\cdot \frac 12\int_{t-\frac 12}^t\int_U|\bar q(x,\tau)|^2 \phi dxd\tau\\
+  \frac {C_2}2\int_{t- 1}^t (G_1(\tau)+G_2(\tau))dxd\tau
=\frac 34  \int_{t-\frac 12}^t\int_U\bar q^2(x,s)\phi dxds
+  \frac {C_2}2\int_{t- 1}^t  (G_1(\tau)+G_2(\tau))d\tau.
\end{multline*}
Hence,\begin{align*}
\int_U\bar q^2(x,t)\phi dx \le 3  \int_{t-\frac 12}^t\int_U\bar q^2(x,s)\phi dxds
+  C\int_{t- 1}^t  (G_1(\tau)+G_2(\tau))d\tau.
\end{align*}
Using \eqref{Gtwo} to bound the first integral on the right-hand side, we obtain 
\beq\label{Gthree}
 \int_U\bar p_t^2(x,t)\phi dx \le C\int_U \bar{p}^2(x,t-1)\phi dx +C\int_{t-1}^t (G(\tau)+G_1(\tau)+G_2(\tau))d\tau.
\eeq

Combining \eqref{EstLtwo} and \eqref{Gthree} gives
\beqs
 \int_U\bar p_t^2(x,t)\phi dx \le C \int_U \bar{p}^2(x,0)\phi(x) dx+ C{\mathcal M}(t-1)^{\frac 2{2-a}} +C\int_{t-1}^t (G(\tau)+G_1(\tau)+G_2(\tau)) d\tau.
\eeqs
Again, by properties \eqref{MG} and \eqref{MA}, estimate  \eqref{G86} follows.

(iii) Taking limit superior of \eqref{Gthree} and using \eqref{noinitial},  we obtain 
\beqs
\limsup_{t\to\infty} \int_U\bar p_t^2(x,t)\phi dx \le C\mathcal A^\frac2{2-a} +C\mathcal A+C\limsup_{t\to\infty}\int_{t-1}^t (G_1(\tau)+G_2(\tau))d\tau,
\eeqs
which yields \eqref{Gthree3}.
The estimate \eqref{limsupq1} follows \eqref{Gthree3} and property \eqref{intlim} for functions $G_1$ and $G_2$ in place of $G$.

(iv) For sufficiently large $t$, estimating the first term on the right-hand side of \eqref{Gthree} by \eqref{EstLtwo7}, and then applying Lemma \ref{Btt} to bound $G(t-1)$ and $G(\tau)$ in terms of $\mathcal B$ and $G(t)$,  we obtain \begin{align*}
 \int_U\bar p_t^2(x,t)\phi dx 
 &\le C(\mathcal B^{\frac 1{1-a}}+ G(t-1)^{\frac 2{2-a}}) +C\int_{t-1}^t (G(\tau)+G_1(\tau)+G_2(\tau))d\tau\\
  &\le C\mathcal B^{\frac 1{1-a}}+ C(1+\mathcal B+G(t))^{\frac 2{2-a}}
  +C(1+\mathcal B+G(t)) 
  +C\int_{t-1}^t (G_1(\tau)+G_2(\tau))d\tau.
\end{align*}
Then  \eqref{Gthree11} follows by simple manipulations using inequalities  \eqref{ee3}, \eqref{ee5}.
\end{proof}

\section{Continuous dependence}\label{dependence}

In this section, we establish the continuous dependence of the solution on the initial and boundary data.


Let $p_1(x,t)$ and  $p_2(x,t)$ be two solutions of \eqref{eq8} with boundary data $\psi_1(x,t)$ and $\psi_2(x,t)$, respectively. 
 For $i=1,2$, let $\Psi_i(x,t)$ be an extension of $\psi_i(x,t)$, and define $\bar p_i=p_i-\Psi_i$. 
Denote 
$$P=p_1-p_2,\ 
\Phi=\Psi_1-\Psi_2\quad \text{and}\quad \bar P=\bar{p}_1-\bar{p}_2=P-\Phi.$$
 Then
\begin{align}
\label{twopeqn}
\phi(x) \frac{\partial \bar P}{\partial{t}}&=\nabla \cdot (K(x,|\nabla p_1|)\nabla  p_1-  K(x,|\nabla p_2|)\nabla p_2) -\phi(x)\Phi_t  \quad \text{on }  U\times(0, \infty),\\
\bar P&=0\quad \text{on }\Gamma \times (0, \infty).\nonumber
\end{align}

The weighted norms of $\bar P$ and $\Phi$ are related by the following differential inequalities.

\begin{lemma}\label{lem61}
For all $t>0$, one has
\beq\label{r1}
\frac d{dt}\int_U \bar P^2(x,t) \phi(x) dx 
\le -d_3 h_1(t)^{-\frac{a}{2-a}} \left(\int_U W_1(x) |\nabla \bar P(x,t)|^{2-a}dx\right)^\frac 2 {2-a}+CD(t)h_2(t)^\frac12,
\eeq
\beq\label{r2}
 \frac d{dt}\int_U \bar P^2(x,t) \phi(x) dx 
\le -d_4 h_1(t)^{-\frac{a}{2-a}}\int_U  \bar P^{2}(x,t)\phi(x) dx+CD(t)h_2(t)^\frac12,
\eeq
where $d_3,d_4>0$,
\begin{align}
\label{Ddef}
D(t)&= \int_U a_0(x)^{-1}|\nabla\Phi(x,t)|^2dx+\Big(\int_U a_0(x)^{-1}|\nabla\Phi(x,t)|^2dx\Big)^\frac12 +\Big(\int_U  |\Phi_t(x,t)|^2\phi(x) dx\Big)^\frac12,\\
\label{h1}
h_1(t)&=B_1+ \sum_{i=1}^2\int_U H(x,|\nabla p_i(x,t)|) dx,\\
\label{h2}
h_2(t)&= 1+\sum_{i=1}^2  \int_U\Big [H(x,|\nabla p_i(x,t)|) +\bar p_i^2(x,t)\phi(x)\Big]dx. 
\end{align}
\end{lemma}

\begin{proof}
We define
\beqs
 D_1(t)=\int_U  |\Phi_t(x,t)|^2\phi (x)dx,\
D_2(t)=\int_U a_0(x)^{-1}|\nabla\Phi(x,t)|^2dx,
\eeqs 
\beqs
h_3(t)=\sum_{i=1}^2 \|\bar{p}_i(t)\|_{L_\phi^2}^2,\quad 
h_4(t)= \sum_{i=1}^2\int_U H(x,|\nabla p_i(x,t)|) dx .
\eeqs

Multiplying equation \eqref{twopeqn} by $\bar P$ and integrating  over $U$ give
\begin{align*}
\int_U \bar P \cdot \bar P_t \phi dx =& \int_U (\nabla\cdot \left(K(x,|\nabla p_1|)\nabla  p_1- K(x,|\nabla p_2|)\nabla p_2\right) ) \bar P dx-\int_U  \Phi_t\, \bar P\phi dx.
\end{align*}
Using integration by parts for the first integral on the right-hand side, we have
\begin{align*}
\frac 12 \frac d{dt}\int_U \bar P^2 \phi dx
& = -\int_U \left(K(x,|\nabla p_1|)\nabla{p_1}-K(x,|\nabla p_2|)\nabla{p_2}\right) \cdot \nabla \bar Pdx-\int_U  \Phi_t \, \bar P\phi dx\\
& = -\int_U  \left(K(x,|\nabla p_1|)\nabla p_1-K(x,|\nabla p_2|)\nabla p_2\right) \cdot (\nabla{p}_1-\nabla{p}_2)dx\\
&\quad+\int_U (K(x,|\nabla p_1|)\nabla p_1-K(x,|\nabla p_2|)\nabla p_2)\cdot\nabla \Phi dx-\int_U  \Phi_t\, \bar P\phi dx.
\end{align*}
Applying Lemma \ref{mono2} to the third to last integrand, we obtain
\begin{align*}
\frac 12 \frac d{dt}\int_U \bar P^2 \phi dx 
&\le -(1-a)\int_U  K(x,|\nabla p_1|\vee|\nabla p_2|) |\nabla{p}_1-\nabla{p}_2|^2 dx\\
&\quad + \int_U (K(x,|\nabla p_1|)|\nabla p_1|+K(x,|\nabla p_2|)|\nabla p_2| ) \, |\nabla \Phi|dx+\int_U  |\Phi_t| |\bar P| \phi dx.
\end{align*}
Above, we use the notation $|\nabla p_1|\vee|\nabla p_2|=\max\{|\nabla p_1|,|\nabla p_2|\}$.

For the first integral on the right-hand side, we note that 
\beqs
|\nabla{p}_1-\nabla{p}_2|^2=|\nabla \bar P +\nabla \Phi|^2 \ge \frac12 |\nabla \bar P|^2-|\nabla \Phi|^2,
\eeqs
hence,
\begin{align}
\frac 12 \frac d{dt}\int_U \bar P^2 \phi dx 
&\le -\frac{1-a}2\int_U  K(x,|\nabla p_1|\vee|\nabla p_2|)| \nabla\bar P|^2 dx+ C \int_U  K(x,|\nabla p_1|\vee|\nabla p_2|)|\nabla \Phi|^2 dx \nonumber \\
&\quad + \int_U (K(x,|\nabla p_1|)|\nabla p_1|+K(x,|\nabla p_2|)|\nabla p_2| ) |\nabla \Phi|dx+\int_U  |\Phi_{t}| |\bar P| \phi dx \nonumber\\
&\eqdef  - \frac{1-a}2 I_1 +CI_2+I_3+I_4. \label{IV}
\end{align}

$\bullet$ Consider $I_1$. Let $\mathcal{K}(x,t)= K(x,|\nabla{p}_1(x,t)|\vee|\nabla{p}_2(x,t)|).$
Then by H\"older's inequality,
\begin{align}\label{Kone}
\int_U W_1(x) |\nabla \bar P|^{2-a}dx\le \left( \int_U \mathcal{K}(x,t)|\nabla \bar P|^2dx \right)^{\frac{2-a}2}J_1^\frac a 2,\quad \text{where  } 
J_1=\int_U \frac{W_1(x)^{\frac 2 a}}{\mathcal{K}(x,t)^{\frac{2-a}{a}}}dx.
\end{align} 
Applying \eqref{KK} to bound $\mathcal K(x,t)$ from below, and then using \eqref{ee3}, we estimate $J_1$ as
\begin{align*}
J_1&\le  \int_U W_1(x)^{\frac 2 a}\Big(\frac{(|\nabla{p}_1|\vee|\nabla{p}_2|)^a+a_N(x)^a}{2W_1(x)}\Big)^\frac{2-a}a dx \\
&\le C \left(\int_U W_1(x)a_N(x)^{2-a}dx+\int_U W_1(x)(|\nabla {p}_1|\vee |\nabla {p}_2|)^{2-a} dx\right)\\
&\le C \left(\int_U W_1(x)a_N(x)^{2-a}dx+\int_U W_1(x)(|\nabla {p}_1|^{2-a}+ |\nabla {p}_2|^{2-a}) dx\right).
\end{align*}
Then by \eqref{W1a} and \eqref{H22}, we have 
\begin{align*}
J_1
&\le C \Big(\int_U a_N(x)dx+\int_U [H(x,|\nabla{p}_1)|+H(x,|\nabla {p}_2|)]dx\Big)
=Ch_1(t).
\end{align*}
This and  \eqref{Kone} yield
\beq\label{EstofIone}
\begin{aligned}
 I_1=\int_U \mathcal{K}(x,t)|\nabla \bar P|^2dx \ge C\left(\int_U W_1(x) |\nabla \bar P|^{2-a}dx\right)^\frac 2 {2-a} h_1(t)^{-\frac{a}{2-a}}.
\end{aligned}
\eeq

$\bullet$ For $I_2$, by using \eqref{decK}
\beq\label{estofItwo}
I_2 \le C\int_U a_0(x)^{-1}|\nabla\Phi|^2dx= C D_2(t).
\eeq

$\bullet$  For $I_3$, applying using H\"older's inequality gives
\begin{align*}
I_3&\le \sum_{i=1,2} \Big\{  \Big(\int_U K(x,|\nabla p_i|)|\nabla p_i|^2 dx \Big)^\frac 12\Big(\int_U K(x,|\nabla p_i|)|\nabla \Phi|^2 dx \Big)^\frac 12\Big\}.
\end{align*}
Using \eqref{H33} for the first integral and  \eqref{decK} the second integral, and then applying Cauchy-Schwarz inequality, we have  
\begin{align*}
I_3&\le\sqrt{2}  \Big(\sum_{i=1,2} \int_U H(x,|\nabla p_i|) dx \Big)^\frac 12 \Big(\int_U a_0(x)^{-1}|\nabla \Phi|^2 dx \Big)^\frac 12.
\end{align*}
Thus,
 \beq\label{estIthre}\begin{aligned}
 I_3 \leq Ch_4(t)^\frac{1}{2} D_2(t)^\frac1{2}.
\end{aligned}
\eeq

$\bullet$  For $I_4$, applying the H\"older's inequality gives
\beq\label{estIfour}
 I_4 \le\|\bar P\|_{L_\phi^2} \| \Phi_t\|_{L_\phi^2}\le ( \|\bar{p}_1\|_{L_\phi^2}+ \|\bar p_2\|_{L_\phi^2}) \| \Phi_t\|_{L_\phi^2}=h_3(t)^\frac12D_1(t)^\frac12.
\eeq

Then combining \eqref{IV}, \eqref{EstofIone}, \eqref{estofItwo}, \eqref{estIthre} and \eqref{estIfour} yields 
\begin{align*}
 \frac d{dt}\int_U \bar P^2 \phi dx 
&\le -d_3\left(\int_U W_1(x) |\nabla \bar P|^{2-a}dx\right)^\frac 2 {2-a} h_1(t)^{-\frac{a}{2-a}}\\
&\quad +CD_2(t)+Ch_4(t)^\frac{1}{2} D_2(t)^\frac1{2}+h_3(t)^\frac12 D_1(t)^\frac12.
\end{align*}
Hence
\beq\label{r3}
 \frac d{dt}\int_U \bar P^2(x,t) \phi(x) dx 
\le -d_3 h_1(t)^{-\frac{a}{2-a}} \left(\int_U W_1(x) |\nabla \bar P(x,t)|^{2-a}dx\right)^\frac 2 {2-a}+CD_5(t),
\eeq
where
\beqs
D_5(t)=D_2(t)+h_4(t)^\frac{1}{2} D_2(t)^\frac1{2}+h_3(t)^\frac12 D_1(t)^\frac12.
\eeqs
We estimate 
\beq\label{D5}
D_5(t)\le C(D_2(t)+D_2(t)^\frac1{2}+ D_1(t)^\frac12)(1 +h_4(t)^\frac{1}{2}+h_3(t)^\frac12)\le CD(t) h_2(t)^\frac12.
\eeq
Therefore, \eqref{r1} follows \eqref{r3} and \eqref{D5}. Finally, using Poincar\'e-Sobolev's inequality \eqref{Poincarenew} for $u=\bar P$, \eqref{r1} implies \eqref{r2}.
\end{proof}

To describe more specific estimates, we introduce
\beqs
\mathcal{\bar P}_0=\sum_{i=1}^2 \int_U \bar{p}_i^2(x,0)\phi(x) dx,\quad
\mathcal{H}_0=\sum_{i=1}^2\int_U H(x,|\nabla p_i(x,0)|)dx,
\eeqs
and, referring to \eqref{G}, \eqref{MG},  \eqref{G8}, \eqref{Gsix},  define for $t\ge 0$
\beqs
\tilde G(t)=\sum_{i=1}^2 G[\Psi_i](t),\quad
\tilde {\mathcal M}(t)=\sum_{i=1}^2  {\mathcal M}[\Psi_i](t),
\eeqs
\beqs
\tilde G_1(t)=\sum_{i=1}^2 G_1[\Psi_i](t),\quad 
\tilde G_2(t)=\sum_{i=1}^2 G_2[\Psi_i](t).
\eeqs

In the following, we show that the $L_\phi^2$-norm of $\bar{P}(t)$ for $t>0$ can be bounded by the initial difference 
$\| \bar{P}(0)\|_{L_\phi^2}$ and the difference between the boundary data  expressed by $D(t)$. It means that the solution of \eqref{ppb} depends continuously on the initial and boundary data.

\begin{theorem}\label{diffestint}
For $t\ge 0$,
\beq\label{yhdfive}
\begin{aligned}
\| \bar{P}(t)\|_{L_\phi^2}^2\le  e^{-d_4\int_0^t {\mathcal M}_1(\tau)^{-\frac{a}{2-a}}d\tau}\| \bar{P}(0)\|_{L_\phi^2}^2+ C\int_0^t e^{-d_4\int_s^t {\mathcal M}_1(\tau)^{-\frac{a}{2-a}}d\tau}{\mathcal M}_1(s)^\frac12D(s) \,ds,
\end{aligned}
\eeq
where 
\beqs
{\mathcal M}_1(t)=\mathcal{H}_0+ \bar{\mathcal{P}}_0+\tilde M(t)^\frac2{2-a} +\sup_{\tau\in[0,t]}\tilde G_1(\tau).
\eeqs

In particular, for any $T>0$,
\beq \label{Pest}
\sup_{t\in[0,T]}\| \bar{P}(t)\|_{L_\phi^2}^2 \le  \| \bar{P}(0)\|_{L_\phi^2}^2 + C {\mathcal M}_1(T)^\frac12\int_0^T D(t)dt.
\eeq
\end{theorem}
\begin{proof}
Define $y(t)=\int_U \bar P^2(x,t)\phi(x) dx$. We rewrite \eqref{r2} as 
\beq\label{yprim}
y'(t) \le -d_4 h_1(t)^{-\frac{a}{2-a}} y(t) + C_0D(t)h_2(t)^\frac12.
\eeq
By Gronwall's inequality
\beq\label{mainP}
y(t)\le y(0)e^{-d_4\int_0^t h_1(\tau)^{-\frac{a}{2-a}}d\tau}+C\int_0^t e^{-d_4\int_s^t h_1(\tau)^{-\frac{a}{2-a}}d\tau} h_2(s)^\frac12D(s)ds.
\eeq

Let $t\ge 0$. Using definition \eqref{h1} of $h_1(t)$  and by applying \eqref{gradpW111} to each $p=p_i$ for $i=1,2$, we have 
\begin{align*}
h_1(t)
&\le B_*+ e^{-\frac14 t} \mathcal{H}_0+ C\bar{\mathcal{P}}_0+C\tilde M^\frac2{2-a}(t) +C\int_0^t e^{-\frac14 (t-\tau)}\tilde G_1(\tau)  d\tau .
\end{align*}
Thus,
\beq\label{h1e}
h_1(t)\le C(\mathcal{H}_0+ \bar{\mathcal{P}}_0+\tilde M^\frac2{2-a}(t) +\sup_{\tau\in[0,t]}\tilde G_1(\tau)) 
= C{\mathcal M}_1(t).
\eeq
Similarly, by \eqref{h2}, estimates \eqref{gradpW111} and \eqref{EstLtwo}, we have 
\begin{align*}
h_2(t)
&\le 1+ e^{-\frac14 t} \mathcal{H}_0+ \mathcal{\bar P}_0+ C\tilde {\mathcal M}(t)^{\frac 2{2-a}}+C\int_0^t e^{-\frac14 (t-\tau)}\tilde G_1(\tau)  d\tau,
\end{align*}
which implies
\beq \label{h4e}
h_2(t)\le C{\mathcal M}_1(t).
\eeq
Therefore, we obtain  \eqref{yhdfive} from \eqref{mainP}, \eqref{h1e} and \eqref{h4e}.

Now, let $T>0$.
Neglecting the exponentials in \eqref{yhdfive} and noting that ${\mathcal M}_1(t)\le {\mathcal M}_1(T)$ for all $t\in[0,T]$, we obtain \eqref{Pest}.
\end{proof}

Next, we estimate $\bar P(t)$  when $t\to\infty$. The estimate is independent of the initial data and only depends on the asymptotic behavior of $\Psi_1(x,t)$, $\Psi_2(x,t)$, and $\Phi(x,t)$ as $t\to\infty$.
 
If $\int_0^\infty h_1(t)^{-\frac{a}{2-a}}dt=\infty$, then from \eqref{yprim} and Lemma \ref{ODE3-lem}(ii) we have
\beq\label{diffR}
\limsup_{t\to\infty}\| \bar{P}(t)\|_{L_\phi^2}^2 \le \frac{C_0}{d_4}\limsup_{t\to\infty}\frac{D(t)h_2(t)^\frac12}{h_1^{-\frac{a}{2-a}}(t)}
= C\limsup_{t\to\infty}R(t),
\eeq
where
\beq\label{Rdef}
 R(t)=  h_2(t)^\frac12h_1(t)^\frac{a}{2-a} D(t).
\eeq

To estimate the last limit in \eqref{diffR}, we define,
referring to  \eqref{AB},  \eqref{G8}, and \eqref{Gsix},  the following numbers
\beqs
\tilde {\mathcal A}=\sum_{i=1}^2 \mathcal A[\Psi_i]=\sum_{i=1}^2\limsup_{t\to\infty} G[\Psi_i](t),\quad 
\tilde{\mathcal B}=\sum_{i=1}^2 \mathcal B[\Psi_i]=\sum_{i=1}^2\limsup_{t\to\infty}[(G[\Psi_i](t))']^-,
\eeqs
\beqs
\mathcal G_1=\sum_{i=1}^2 \limsup_{t\to\infty} G_1[\Psi_i](t),\quad 
\mathcal G_2=\sum_{i=1}^2 \limsup_{t\to\infty} G_2[\Psi_i](t).
\eeqs

The asymptotic behavior  of $\Phi(x,t)$ as $t\to\infty$ will be characterized by
\beqs
\mathcal D=\limsup_{t\to\infty}D(t).
\eeqs

Denote also that
\beqs
\kappa_0=\frac{a}{2-a}+\frac12=\frac{2+a}{2(2-a)}.
\eeqs
\begin{theorem}\label{theo53}
If $\tilde {\mathcal A}$ and $\mathcal G_1$ are finite, then 
\beq \label{limest}
\limsup_{t\to\infty}\| \bar{P}(t)\|_{L_\phi^2}^2 
\le C(\tilde {\mathcal A}^\frac{2}{2-a}+\mathcal G_1)^{\kappa_0}\mathcal D.
\eeq
\end{theorem} 
\begin{proof}
Note from \eqref{noinitial} and \eqref{gradpWone} that
\beq\label{limh1}
 \limsup_{t\to\infty} h_1(t), 
 \limsup_{t\to\infty} h_2(t)\le C (\tilde {\mathcal A}^\frac{2}{2-a} +\mathcal G_1)<\infty.
\eeq
Then  $h_1(t)$ and $h_2(t)$ are bounded on $[0,\infty)$.
Thus, $\int_0^\infty h_1(t)^{-\frac{a}{2-a}}dt=\infty$ and, consequently, estimate \eqref{diffR} holds.
By \eqref{Rdef} and  \eqref{limh1},
\beq\label{Rlim}
 \limsup_{t\to\infty} R(t) 
\le C(\tilde {\mathcal A}^\frac{2}{2-a} +\mathcal G_1)^\frac12 (\tilde {\mathcal A}^\frac{2}{2-a} +\mathcal G_1)^\frac{a}{2-a} \mathcal D=C(\tilde {\mathcal A}^\frac{2}{2-a}+\mathcal G_1)^{\kappa_0} \mathcal D.
\eeq
Therefore, \eqref{limest} follows this and \eqref{diffR}.
\end{proof}

Now, we focus on the case when the boundary data is unbounded as $t\to\infty$.

$\bullet$  If $t>0$ then, by \eqref{EstLtwo},
\beq\label{pM12}
\sum_{i=1}^2 \int_U \bar p_i^2(x,t)\phi(x) dx \le C\Big(\mathcal{\bar P}_0+\tilde {\mathcal M}(t)^\frac{2}{2-a}\Big).
\eeq

If $t\ge 1$ then, by \eqref{newgrad}, 
\beq\label{newgrad66}
\sum_{i=1}^2 \int_U H(x,|\nabla p_i(x,t)|)dx\le C\Big(\mathcal{\bar P}_0+\tilde {\mathcal M}(t)^\frac{2}{2-a} + \int_{t-1}^t \tilde G_1(\tau)d\tau\Big),
\eeq
and, by \eqref{G86},
\beq\label{G8666}
 \sum_{i=1}^2 \int_U \bar p_{i,t}^2(x,t)\phi(x) dx \le C\Big(\bar{\mathcal P}_0 + \tilde {\mathcal M}(t)^{\frac 2{2-a}}  +\int_{t-1}^t (\tilde G_1(\tau)+\tilde G_2(\tau))d\tau\Big).
\eeq

$\bullet$ In case  $\tilde{\mathcal B}<\infty$, then  ${\mathcal B}[\Psi_1]$ and ${\mathcal B}[\Psi_2]$ are finite. Using estimates \eqref{EstLtwo7}, \eqref{Gtwo33}, \eqref{Gthree11} for $p_i$ with $i=1,2$,  there is $T_0>0$ such that 
for $t>T_0$, one has
\beq\label{p12}
\sum_{i=1}^2 \int_U \bar p_i^2(x,t)\phi (x) dx \le C\Big(\tilde{\mathcal B}^{\frac 1{1-a}}+ \tilde G(t)^{\frac 2{2-a}}\Big),
\eeq
\beq\label{Gtwo3366}
\sum_{i=1}^2 \int_U H(x,|\nabla p_i(x,t)|)dx\le  C\Big(\tilde{\mathcal B}^{\frac 1{1-a}}+ \tilde G(t)^{\frac 2{2-a}}+\int_{t-1}^t \tilde G_1(\tau)d\tau\Big),
\eeq
 \beq\label{Gthree1166}
\sum_{i=1}^2 \int_U \bar p_{i,t}^2(x,t)\phi(x) dx \le C\Big(\tilde{\mathcal B}^{\frac 1{1-a}}+ \tilde G(t)^{\frac 2{2-a}} +\int_{t-1}^t (\tilde G_1(\tau)+\tilde G_2(\tau))d\tau\Big).
\eeq

$\bullet$ Assume $\tilde {\mathcal A}=\infty$.
Then $\lim_{t\to\infty}\tilde {\mathcal M}(t)=\infty$, and for sufficiently large $t$, one has $\tilde {\mathcal M}(t)\ge \bar{\mathcal P}_0$.

From \eqref{newgrad66}, \eqref{G8666}, and \eqref{Gtwo3366}, \eqref{Gthree1166}, we have  for large $t$ that 
\beq\label{ptHV}
\sum_{i=1}^2 \int_U H(x,|\nabla p_i(x,t)|)dx,\
\sum_{i=1}^2 \int_U \bar p_{i,t}^2(x,t)\phi(x) dx  \le CV(t),
\eeq
where
\beq\label{Vdef}
V(t)= \begin{cases}
\tilde {\mathcal M}(t)^{\frac 2{2-a}} +\int_{t-1}^t (\tilde G_1(\tau)+\tilde G_2(\tau)) d\tau&\text{in general,}\\
\\
\tilde{\mathcal B}^{\frac 1{1-a}}+ \tilde G(t)^{\frac 2{2-a}} +\int_{t-1}^t (\tilde G_1(\tau)+\tilde G_2(\tau))d\tau& \text{when }\tilde{\mathcal B}<\infty.
\end{cases}
\eeq

With the above preparations, we are ready to estimate $\| \bar{P}(t)\|_{L_\phi^2}$ as $t\to\infty$ in the case $\tilde {\mathcal A}=\infty$ . 

\begin{theorem}\label{theo54}
Assume $\tilde {\mathcal A}=\infty$. If $\int_1^\infty V(t)^{-\frac{a}{2-a}}dt=\infty$, then 
\beq\label{limP}
\limsup_{t\to\infty} \| \bar{P}(t)\|_{L_\phi^2}^2 \le  C\limsup_{t\to\infty} \Big( V(t)^{\kappa_0} D(t)\Big).
\eeq
\end{theorem}
\begin{proof}
By \eqref{p12} and  \eqref{Gtwo3366}, or \eqref{pM12} and \eqref{newgrad66},  we have  for large $t$
\beq\label{hV} 
h_2(t),  h_1(t)\le CV(t).
\eeq
Combining this with \eqref{diffR}, we have 
\begin{align*}
\limsup_{t\to\infty} \| \bar{P}(t)\|_{L_\phi^2}^2  
&\le C\limsup_{t\to\infty}( h_2(t)^\frac12h_1(t)^\frac{a}{2-a}D(t)) 
\le C\limsup_{t\to\infty} ( V(t)^\frac12 V(t)^\frac{a}{2-a}D(t))\\
&= C\limsup_{t\to\infty} ( V(t)^{\kappa_0}D(t)).
\end{align*}
This proves \eqref{limP}.
\end{proof}

The estimate \eqref{limP} can be  interpreted as follows. As $t\to\infty$, even though $V(t)\to\infty$, if the boundary data's difference characterized by $D(t)$ decays  very fast, it can diminish the growth of $V(t)$ and result in $\| \bar{P}(t)\|_{L_\phi^2}$ going to zero. 

Now, we turn to the continuous dependence for the pressure gradient. What the results obtained below mean for $\| \nabla \bar{P}(t)\|_{L_{W_1}^{2-a}}$ are the same as Theorems \ref{diffestint}, \ref{theo53}, and \ref{theo54} for $\| \bar{P}(t) \|_{L_\phi^2}$.

\begin{theorem}\label{theo55}
Let $t_0\in(0,1]$. For $t\ge t_0>0$,
\begin{multline}\label{keygrad0}
\| \nabla \bar{P}(t)\|_{L_{W_1}^{2-a}}^2 
\le C {\mathcal M}_2(t)^{\kappa_0}
\Big( e^{-d_4\int_0^t {\mathcal M}_1(\tau)^{-\frac{a}{2-a}}d\tau} \| \bar{P}(0) \|_{L_\phi^2}^2+ D(t)^2\\
+ \int_0^t e^{-d_4\int_s^t {\mathcal M}_1(\tau)^{-\frac{a}{2-a}}d\tau}{\mathcal M}_1(s)^\frac12 D(s) \,ds \Big)^\frac 12,
\end{multline}
where
\beqs
{\mathcal M}_2(t)=t_0^{-1} ( \mathcal{H}_0+\bar{\mathcal{P}}_0)+ \tilde {\mathcal M}(t)^\frac{2}{2-a} +\sup_{\tau\in [0,t]} (\tilde G_1(\tau)+ \tilde G_2(\tau)).
\eeqs
Moreover,
\beq\label{limdg}
\limsup_{t\to\infty} \| \nabla \bar{P}(t)\|_{L_{W_1}^{2-a}}^2
 \le C\Big[(\tilde {\mathcal A}^\frac2{2-a}+\mathcal G_1+\mathcal G_2)^{3\kappa_0} \mathcal D\Big]^\frac12
+C(\tilde {\mathcal A}^\frac{2}{2-a}+\mathcal G_1)^{\kappa_0}\mathcal D.
\eeq
\end{theorem}
\begin{proof}
Multiplying  \eqref{r1} by $d_3^{-1} h_1(t)^{\frac{a}{2-a}}$, we have
\begin{align*}
 \left(\int_U W_1(x) |\nabla \bar P(x,t)|^{2-a}dx\right)^\frac 2 {2-a} 
&\le -d_3^{-1}h_1(t)^{\frac{a}{2-a}} \frac d{dt}\int_U \bar P^2(x,t) \phi(x) dx +Ch_1(t)^{\frac{a}{2-a}}D(t)h_2(t)^\frac12\\
&\le Ch_1(t)^{\frac{a}{2-a}}\int_U |\bar P||\bar P_t|\phi dx +CR(t)\\
&\le Ch_1(t)^{\frac{a}{2-a}} \|\bar P_t\|_{L^2_\phi}\|\bar P\|_{L^2_\phi}+CR(t).
\end{align*}
%
%
Applying  triangle inequality to $ \|\bar P_t\|_{L^2_\phi}$ gives
\beq\label{keygrad}
\| \nabla \bar{P}(t)\|_{L_{W_1}^{2-a}}^2
\le Ch_1(t)^{\frac{a}{2-a}}(\|\bar p_{1,t}\|_{L^2_\phi}+\|\bar p_{2,t}\|_{L^2_\phi})\|\bar P\|_{L^2_\phi}+CR(t).
\eeq

By \eqref{Rdef}, \eqref{h1e} and \eqref{h4e},
\beq\label{RV}
R(t)\le C{\mathcal M}_1(t)^\frac12 {\mathcal M}_1(t)^{\frac{a}{2-a}}D(t)=C {\mathcal M}_1(t)^{\kappa_0} D(t).
\eeq

By \eqref{greatert0},  
\beq\label{q1122}
\begin{aligned}
&\|\bar p_{1,t}\|_{L^2_\phi(U)}+\|\bar p_{2,t}\|_{L^2_\phi(U)}
\le C\Big( t_0^{-1} \int_0^{t_0}  \tilde G_1(\tau)d\tau+t_0^{-1} ( \mathcal{H}_0+\bar{\mathcal{P}}_0)\\
&\quad + \tilde {\mathcal M}(t)^\frac{2}{2-a} + \int_{0}^t e^{-\frac14(t-\tau)} [\tilde G_2(\tau)+ \tilde G_1(\tau)]d\tau\Big)^\frac 12
\le C{\mathcal M}_2(t)^\frac12.
\end{aligned}
\eeq

Combining estimates  \eqref{h1e},  \eqref{q1122},  \eqref{yhdfive}, \eqref{RV} with \eqref{keygrad} yields
\begin{multline*}
\| \nabla \bar{P}(t)\|_{L_{W_1}^{2-a}}^2 \le 
C{\mathcal M}_1(t)^{\frac{a}{2-a}}{\mathcal M}_2(t)^\frac12
\cdot\Big \{e^{-d_4\int_0^t {\mathcal M}_1(s)^{-\frac{a}{2-a}}ds} \| \bar{P}(0)\|_{L_\phi^2}^2\\
+ \int_0^t e^{-d_4\int_s^t {\mathcal M}_1(\tau)^{-\frac{a}{2-a}}d\tau}{\mathcal M}_1(s)^\frac12D(s) \,ds \Big \}^\frac 12
+C {\mathcal M}_1(t)^{\kappa_0}D(t).
\end{multline*}
Estimating the first and last ${\mathcal M}_1(t)$ terms on the right-hand side by ${\mathcal M}_1(t)\le {\mathcal M}_2(t)$,
we obtain  \eqref{keygrad0}.

Let $\mathcal D_0=(\tilde {\mathcal A}^\frac{2}{2-a}+\mathcal G_1)^{\kappa_0} \mathcal D$.
Taking limit superior of \eqref{keygrad}, and using the limit estimates \eqref{limh1}, \eqref{limest},  \eqref{Rlim} and  \eqref{limsupq1}, we have 
\begin{align*}
&\limsup_{t\to\infty}\| \nabla \bar{P}(t)\|_{L_{W_1}^{2-a}}^2
\le C(\tilde {\mathcal A}^\frac2{2-a}+\mathcal G_1)^\frac a{2-a} (\tilde {\mathcal A}^\frac2{2-a}+\mathcal G_1+\mathcal G_2)^\frac12  \mathcal D_0^\frac12 +C\mathcal D_0\\
&\le C(\tilde {\mathcal A}^\frac2{2-a}+\mathcal G_1+\mathcal G_2)^\frac{3\kappa_0}{2}  \mathcal D^\frac12
+C(\tilde {\mathcal A}^\frac{2}{2-a}+\mathcal G_1)^{\kappa_0}\mathcal D,
\end{align*}
hence obtaining \eqref{limdg}.
\end{proof}

Finally, we derive the gradient estimates for the case $\tilde{\mathcal A}=\infty$.

\begin{theorem}
Assume $\tilde{\mathcal A}=\infty$. Let $V(t) $ be defined by \eqref{Vdef}.
Suppose 
\beq \label{Vcond}
\int_1^\infty V^{-\frac{a}{2-a}}(t)dt=\infty\quad\text{and}\quad \lim_{t\to\infty}(V^\frac{a}{2-a}(t))'=0.
\eeq
Then
\beq\label{limgP}
\limsup_{t\to\infty} \| \nabla \bar{P}(t)\|_{L_{W_1}^{2-a}}^2
 \le C \limsup_{t\to\infty}\Big[V(t)^{3\kappa_0}D(t)\Big]^\frac12
+ C\limsup_{t\to\infty}\Big[V(t)^{\kappa_0}D(t)\Big].
\eeq
\end{theorem}
\begin{proof}
By \eqref{keygrad}, \eqref{hV}, \eqref{ptHV} and \eqref{Rdef}, we have  for large $t$ that
\beqs
\| \nabla \bar{P}(t)\|_{L_{W_1}^{2-a}}^2 \le CV(t)^{\kappa_0}\|\bar P\|_{L^2_\phi}+CV(t)^{\kappa_0}D(t).
\eeqs
Taking limit superior of the previous  inequality  yields 
\beq\label{kgss}
\limsup_{t\to\infty}\| \nabla \bar{P}(t)\|_{L_{W_1}^{2-a}}^2 \le C\limsup_{t\to\infty}\left( V(t)^{2\kappa_0}\int_U|\bar P|^2\phi dx \right)^\frac{1}{2}+C\limsup_{t\to\infty}(V(t)^{\kappa_0}D(t)).
\eeq

Consider first  limit on the right-hand side of the \eqref{kgss}. 
Let $y(t)=\int_U|\bar P(x,t)|^2\phi(x) dx$.
By  \eqref{yprim} and \eqref{hV},  we have for large $t$ that
\beqs
y'(t) \le -C_1 V(t)^{-\frac{a}{2-a}} y(t) + C_2D(t)V(t)^\frac12.
\eeqs

We apply  Lemma \ref{multdif} to $h(t)=C_1V^{-\frac{a}{2-a}}(t)$, $f(t)=C_2D(t)V(t)^\frac12$, and $g(t)=V^{2\kappa_0}(t)$, noticing that condition \eqref{A2cond} is met thanks to \eqref{Vcond}.
It follows that
\beqs
\limsup_{t\to\infty}\Big(V(t)^{2\kappa_0}y(t)\Big) 
\le C \limsup_{t\to\infty}(V(t)^{2\kappa_0} V(t)^\frac{a}{2-a} V(t)^\frac12 D(t))
=C \limsup_{t\to\infty}(V(t)^{3\kappa_0}D(t)).
\eeqs
Then inequality \eqref{limgP} follows  this and \eqref{kgss}.
\end{proof}

\textbf{Acknowledgement.}  The authors would like to thank N.~C.~Phuc for helpful discussions. L.~H. acknowledges the support by NSF grant DMS-1412796.


 \appendix
 
  
 \section{Appendix} \label{append}

 We collect here some useful lemmas on solutions of differential inequalities.

\begin{lemma}[c.f. \cite{HIKS1}, Lemma A.1]\label{ODE3-lem}
 Let $\phi$ be a continuous, strictly increasing function from $[0,\infty)$ onto $[0,\infty)$.
 Suppose $y(t)\ge 0$ is a continuous function on $[0,\infty)$ such that
 \beqs
 y'(t)\le -h(t)\phi^{-1}(y(t)) +f(t)\quad \forall t>0,
 \eeqs
  where $h(t)>0$, $f(t)\ge 0$  are continuous functions on $[0,\infty)$.
\begin{enumerate}
\item[\rm (i)] If ${\mathcal M}(t)$ is an increasing, continuous function on $[0,\infty)$ that satisfies ${\mathcal M}(t)\ge f(t)/h(t)$ for all $t\ge 0$,   then
  \beqs
  y(t)\le y(0)+\phi({\mathcal M}(t)) \quad \forall t\ge 0.
  \eeqs

\item[\rm (ii)] If $\int_0^\infty h(\tau)d\tau=\infty$ then
  \beqs
  \limsup_{t\rightarrow\infty} y(t)\le \phi\left( \limsup_{t\rightarrow\infty}\frac{f(t)}{h(t)}\right).
  \eeqs
Here, we use the notation $\phi(\infty)=\infty$.
\end{enumerate}
  \end{lemma}

\begin{lemma}[c.f.~\cite{HI2}, Proposition 3.7]\label{supinf}
Let $\phi(z)=c(z+z^\gamma)$ for all $z\ge 0$, where $c>0$ and $1<\gamma< 2$. 
Suppose $y(t)\ge 0$ is a continuous function on $[0,\infty)$ such that
 \beqs
 y'(t)\le -\phi^{-1}(y(t)) +f(t)\quad \forall t>0,
 \eeqs
  where  $f(t)\ge 0$  is a function in $C([0,\infty))\cap C^1((0,\infty))$.
  
  Assume $\beta\eqdef \limsup_{t\to\infty}[f'(t)]^-$ is finite.
Then there is $T>0$ such that
\beq\label{correctineq} 
y(t)\le C\big(1+\beta^\frac{\gamma}{2-\gamma}+f(t)^\gamma\big )\text{ for all } t>T,
\eeq
where 
$C=3[32(1+c)]^\frac2{2-\gamma}.$
\end{lemma}
\begin{proof}
We track and calculate the constant $C$ explicitly. From (3.19) in the proof of Proposition 3.7 \cite{HI2},
\beqs
y(t)
\le \phi(2f(t)+(1+8c\beta\gamma)^\frac1{2-\gamma})
\le 2c(1+2f(t)+(1+\beta)^\frac{1}{2-\gamma} [16(1+ c)]^\frac1{2-\gamma})^\gamma.
\eeqs
Estimating $(1+\beta)^\frac{1}{2-\gamma} $ by \eqref{ee3}, we have 
\begin{align*}
y(t)
&\le 2c(1+2f(t)+2^{\frac{1}{2-\gamma}-1} (1+\beta^\frac{1}{2-\gamma} )[16(1+ c)]^\frac1{2-\gamma})^\gamma
\le 2c[32(1+ c)]^\frac\gamma{2-\gamma}(f(t)+1+\beta^\frac{1}{2-\gamma})^\gamma\\
&\le  2^\frac{2(2\gamma+1)}{2-\gamma}(1+ c)^\frac2{2-\gamma}3^{\gamma-1} \big(1+\beta^\frac{\gamma}{2-\gamma}+f(t)^\gamma\big )
\le 3\cdot [2^5(1+ c)]^\frac2{2-\gamma} \big(1+\beta^\frac{\gamma}{2-\gamma}+f(t)^\gamma\big ),
\end{align*}
which proves \eqref{correctineq}.
\end{proof}

\begin{lemma}\label{multdif}
Let $T\in\mathbb R$. Suppose the continuous functions $y(t),f(t)\ge 0$ and $h(t),g(t)>0$ on $[T,\infty)$ satisfy 
\beqs
y'(t)\le -h(t) y(t)+f(t)\quad \forall t>T,
\eeqs
\beq\label{A2cond}
\int_T^\infty h(\tau)d\tau=\infty\quad\text{and}\quad  \lim_{t\to\infty} \frac{g'(t)}{g(t)h(t)}=0,
\eeq
then
\beqs
\limsup_{t\to\infty}(g(t)y(t))\le \limsup_{t\to\infty}\Big(\frac{g(t)f(t)}{h(t)}\Big).
\eeqs
\end{lemma}
\begin{proof}
 Same as Lemma A.3 of \cite{HKP1}.
\end{proof}

\begin{lemma}\label{Btt}
 Let $f(t)\ge 0$ be a $C^1$-function on $(0,\infty)$. Assume
\beqs
\beta=\limsup_{t\to\infty} [f'(t)]^-<\infty.
\eeqs
Then there is $T>0$ such that for any $t_2>t_1>T$,
\beq\label{A4}
f(t_1)\le f(t_2)+(t_2-t_1)(\beta+1).
\eeq
\end{lemma}
\begin{proof}
 There exists $T>0$ such that for all $t>T$ one has
$
-f'(t)\le \beta+1.
$
Let $t_2>t_1>T$. Then 
\beqs
f(t_1)=f(t_2)-\int_{t_1}^{t_2} f'(\tau)d\tau\le f(t_2)+\int_{t_1}^{t_2} (\beta+1)d\tau=f(t_2)+(t_2-t_1)(\beta+1),
\eeqs
which proves \eqref{A4}.
\end{proof}



\myclearpage

\bibliographystyle{abbrv}

\def\cprime{$'$}

\end{document}